\documentclass[12pt]{amsart}
\pdfoutput=1
\usepackage[foot]{amsaddr}

\usepackage[utf8]{inputenc}
\usepackage[english]{babel}
\usepackage{indentfirst}
\usepackage{amsmath}
\usepackage{amsthm}
\usepackage{amssymb}
\usepackage{amsfonts}
\usepackage{microtype}
\usepackage{fancyhdr}
\usepackage{tikz-cd}
\usetikzlibrary{arrows}
\usetikzlibrary{positioning}
\usepackage{xcolor}
\usepackage[colorlinks]{hyperref}
\usepackage{mathtools}
\usepackage{thmtools}

\usepackage{import}
\usepackage{xifthen}
\usepackage{pdfpages}
\usepackage{transparent}
\usepackage[
  margin=1in,
  includefoot,
  footskip=30pt,
]{geometry}

\usepackage{csquotes}

\usepackage[style=alphabetic, citestyle=alphabetic]{biblatex}
\hypersetup{
  colorlinks = true,
  citecolor = {purple!70!black},
  linkcolor = {red!60!black}
}

\newcommand{\kk}{\Bbbk}
\newcommand{\OO}{\mathcal{O}}
\newcommand{\PP}{\mathbb{P}}
\newcommand{\GG}{\mathbb{G}}

\newcommand{\ZZ}{\mathbb{Z}}

\newcommand{\Ext}{\mathrm{Ext}}

\newcommand{\Hom}{\mathrm{Hom}}

\newcommand{\RHom}{\mathrm{RHom}}

\newcommand{\intRHom}{\mathrm{R}\mathcal{H}\mathrm{om}}

\newcommand{\Pic}{\mathrm{Pic}}

\newcommand{\Coh}{\mathrm{Coh}}
\newcommand{\Dbcoh}{D^b_{\!\mathrm{coh}}}

\newcommand{\Perf}{\mathrm{Perf}}

\newcommand{\cA}{\mathcal{A}}
\newcommand{\cB}{\mathcal{B}}

\newcommand{\emptyperp}{{}^\perp}
\newcommand{\infinity}{\infty}

\newcommand{\qis}{\operatorname{qis}}
\newcommand{\dual}{{\scriptstyle\vee}}
\newcommand{\iso}{\simeq}
\newcommand{\caniso}{\cong}
\newcommand{\isoarrow}{\xrightarrow{\sim}}
\newcommand{\monoarrow}{\hookrightarrow}
\newcommand{\epiarrow}{\twoheadrightarrow}

\newcommand{\rk}{\operatorname{rk}}

\newcommand{\supp}{\operatorname{supp}}

\declaretheoremstyle[
headformat=\NUMBER.\,\NAME\NOTE,
postheadspace=.5em,
spaceabove=6pt,
headfont=\normalfont\small\scshape,
notefont=\normalfont\small\mdseries, notebraces={(}{)},
bodyfont=\normalfont\itshape
]{plainswap}
\declaretheoremstyle[
headformat=\NAME\NOTE,
postheadspace=.5em,
spaceabove=6pt,
headfont=\normalfont\small\scshape,
notefont=\normalfont\small\mdseries, notebraces={(}{)},
bodyfont=\normalfont\itshape
]{nonumplainswap}
\declaretheoremstyle[
headformat=\NUMBER.\,\NAME\NOTE,
postheadspace=.5em,
spaceabove=6pt,
headfont=\normalfont\small\scshape,
notefont=\normalfont\mdseries, notebraces={(}{)},
bodyfont=\normalfont
]{definitionswap}
\declaretheoremstyle[
headformat=\NAME\NOTE,
postheadspace=.5em,
spaceabove=6pt,
headfont=\normalfont\itshape,
notefont=\mdseries, notebraces={(}{)},
bodyfont=\normalfont
]{myremark}

\declaretheorem[style=plainswap, name=Theorem, sharenumber=subsection]{theorem}

\declaretheorem[style=plainswap, numberlike=theorem, name=Proposition]{proposition}

\declaretheorem[style=plainswap, numberlike=theorem, name=Lemma]{lemma}
\declaretheorem[style=plainswap, numberlike=theorem, name=Corollary]{corollary}

\declaretheorem[style=plainswap, numberlike=theorem, name=Conjecture]{conjecture}

\theoremstyle{definition}
\declaretheorem[style=definitionswap, numberlike=theorem, name=Definition]{definition}

\declaretheorem[style=definitionswap, numberlike=theorem, name=Example]{example}
\declaretheorem[style=definitionswap, numberlike=theorem, name=Remark]{remark}
\theoremstyle{myremark}

\numberwithin{equation}{theorem}

\usepackage{quiver}

\title{Admissible subcategories supported on curves}
\author{Dmitrii Pirozhkov$^1$}
\address{$^1$Steklov Mathematical Institute of Russian Academy of Sciences, Moscow, Russia}
\email{dpirozhkov@mi-ras.ru}

\begin{document}

\begin{abstract}
  Let $X$ be a smooth projective variety. We study admissible subcategories of the bounded derived category of coherent sheaves on $X$ whose support is a proper subvariety $Z \subset X$. We show that any one-dimensional irreducible component of $Z$ is a rational curve. When $\operatorname{dim} Z = 1$, we prove that at least one irreducible component in $Z$ intersects the canonical class $K_X$ negatively. In particular, this implies that a surface with a nef and effective canonical bundle has indecomposable derived category, confirming the conjecture by Okawa. We also prove that a configuration of curves with non-negative self-intersections on a surface cannot support an admissible subcategory.
\end{abstract}
\maketitle

\section{Introduction}

The study of semiorthogonal decompositions for derived categories of coherent sheaves on algebraic varieties is an important part of modern algebraic geometry~\cite{kuznetsov-survey}. One important problem in this area is to understand which varieties have indecomposable derived categories. Another core problem, which unfortunately is out of reach for currently existing methods, is to classify semiorthogonal decompositions (equivalently: admissible subcategories) for a given variety when non-trivial decompositions exist. Both these problems are easy to handle for smooth proper curves~\cite{okawa-curve}, but very difficult already in the next natural class of varieties, smooth projective surfaces. On surfaces even the simplest and most studied type of a semiorthogonal decomposition, namely exceptional collections, is complicated in full generality. Despite many strong results~\cite{gorodentsev-rudakov, kuleshov-orlov, okawa-uehara}, exceptional collections on surfaces still hide many secrets~\cite{krah-phantom}.

In this paper we study admissible subcategories in the ``intermediate'' dimension case: we consider semiorthogonal decompositions for higher dimensional varieties, but not arbitrary decompositions, only those where one of the components is a subcategory whose support is one-dimensional. The support condition can often be deduced from the properties of the canonical line bundle: Kawatani and Okawa proved \cite[Thm.~1.1]{kawatani-okawa} that for a smooth projective variety $X$ in any semiorthogonal decomposition at least one of the components is supported on the base locus of the canonical linear system. Thus whenever the base locus of~$K_X$ is one-dimensional, our results give further constraints on the possible semiorthogonal decompositions.

The main results we prove in this paper are collected in the following theorem. For the notion of support of an admissible subcategory see Definition~\ref{def:support_subcategory}.

\begin{theorem}
  \label{thm:collected_theorem}
  Let $X$ be a smooth projective variety over an algebraically closed field $\kk$ of characteristic zero, and let $\cA \subset \Dbcoh(X)$ be an admissible subcategory. Assume that the set-theoretical support $Z := \supp(\cA)$ of the subcategory $\cA$ is not equal to $X$. Then:
  \begin{itemize}
  \item if $C \subset Z$ is an irreducible component with $\dim C = 1$, then $C$ is a rational curve, i.e., the geometric genus of $C$ is zero;
  \item if $\dim Z = 1$, then at least one irreducible component $C \subset Z$ satisfies~$K_X \cdot C < 0$;
  \item if $\dim Z = 1$ and $\dim X = 2$, then at least one of the irreducible components $C \subset Z$ satisfies $C \cdot C < 0$.
  \end{itemize}
\end{theorem}

The first claim is proved in Theorem~\ref{thm:rational_curves_only}; note that it applies even when $Z$ may have other irreducible components of dimension larger than one. The second claim is the most complicated one and it is proved in Theorem~\ref{thm:no_nefness_on_support}. The third claim is proved by classical arguments in Theorem~\ref{thm:sod_implies_negativity}. Only the second claim uses the assumption that the characteristic is zero.

As we explain in Section~\ref{sec:negativity_for_canonical}, the second claim of Theorem~\ref{thm:collected_theorem}, combined with Kawatani--Okawa canonical base locus result \cite[Thm.~1.1]{kawatani-okawa}, establishes indecomposability for derived categories of certain varieties:

\begin{corollary}[{see Corollary~\ref{cor:nef_with_curve_base_locus}}]
  \label{cor:nef_with_curve_base_locus_intro}
  Let $X$ be a smooth projective variety. If the canonical bundle~$K_X$ is nef and the base locus of $|K_X|$ has dimension at most one, then $X$ admits no non-trivial semiorthogonal decompositions.
\end{corollary}

This represents significant progress towards the following conjecture:

\begin{conjecture}[{``well-known'' according to \cite[Conj.~1.6]{bgl-symmetric-curves}}]
  Let $X$ be a smooth projective variety. If the canonical bundle~$K_X$ is nef and effective, then $X$ admits no non-trivial semiorthogonal decompositions.
\end{conjecture}

Note that the effectivity of $K_X$ is equivalent to the dimension of the base locus of $K_X$ being strictly smaller than $\dim X$. In the case where $\dim \mathrm{Bs} |K_X|$ is zero the conjecture is proved in \cite[Cor.~1.3]{kawatani-okawa}, and now Corollary~\ref{cor:nef_with_curve_base_locus_intro} covers the case where $\dim \mathrm{Bs} |K_X|$ is one. This corollary also completes the proof of the following statement about surfaces, conjectured by Okawa in \cite[Conj.~1.8]{okawa-irregular}:

\begin{theorem}[{see Theorem \ref{thm:minimal_surfaces_indecomposable}}]
  \label{thm:minimal_surfaces_indecomposable_intro}
  A smooth projective surface~$X$ with nef canonical class~$K_X$ admits a non-trivial semiorthogonal decomposition of $\Dbcoh(X)$ if and only if $\OO_X$ is an exceptional object, i.e., if $H^1(\OO_X) = H^2(\OO_X) = 0$.
\end{theorem}

The key technical novelty of this paper is that we apply the theory of Picard schemes of (possibly non-reduced, reducible) proper curves to put constraints on admissible subcategories supported on those curves (Theorem~\ref{thm:admissible_invariance}). Thus we generalize Kawatani--Okawa invariance theorem \cite[Thm.~1.4]{kawatani-okawa}, which says that if~$X$ is a smooth projective variety, then any admissible subcategory~$\cA \subset \Dbcoh(X)$ is preserved by the action of the connected component~$\Pic^0(X)$  on~$\Dbcoh(X)$. That theorem is proved by considering a family of admissible subcategories in $\Dbcoh(X)$ parametrized by the scheme $\Pic^0(X)$. 

We specialize to the case where the admissible subcategory $\cA \subset \Dbcoh(X)$ is (set-theoretically) supported on a subvariety $Z \subset X$. The main complication, as compared to the original Kawatani--Okawa setting, is that instead of a family of admissible subcategories in $\Dbcoh(X)$ we get, at best, a family of functors $\cA \to \Dbcoh(X)$ parametrized by $\Pic^0(\widetilde{Z})$ for some infinitesimal thickening $Z \subset \widetilde{Z}$. Thus both the proof and the statement of the invariance we prove are more complicated than in \cite{kawatani-okawa} or in other known generalizations~\cite{lin-indecomposability, caucci-indecomposability}. The details are explained in Section~\ref{sec:invariance}.

\textbf{Structure of the paper}. In Section~\ref{sec:basics} we recall the definition of admissible subcategories and some of their properties. Section~\ref{sec:supported_objects} contains several observations about objects and admissible subcategories in derived categories of coherent sheaves which are supported on a proper subvariety. In Section~\ref{sec:picard} we discuss the Picard scheme and define \textit{infinitely extendable} line bundles on a subvariety. Section~\ref{sec:invariance} is the technical core of the paper: in this section we strengthen the $\Pic^0$-rigidity of admissible subcategories proved in \cite{kawatani-okawa} for admissible subcategories supported on a subvariety by using infinitely extendable line bundles on some infinitesimal neighborhood of that subvariety. As an application, in Section~\ref{sec:rigidity} we prove the first claim of Theorem~\ref{thm:collected_theorem} by considering the projections of skyscraper sheaves. Then, in Section~\ref{sec:negativity_for_canonical}, we use the rigidity to show that a configuration of curves cannot support an admissible subcategory if all of them intersect the canonical class non-negatively. Section~\ref{sec:self_intersections_negative} does not depend on the previous sections; there we show that a configuration of curves in a surface cannot support an admissible subcategory if all those curves have non-negative self-intersection.

\textbf{Notation}. All schemes and triangulated categories in this paper are over a field $\kk$. All functors are assumed to be derived functors between derived categories of coherent sheaves, unless otherwise noted.

\textbf{Acknowledgments}. I thank Alexander Kuznetsov for helpful discussions.

\textbf{Funding information}. This work is supported by the Russian Science Foundation grant no.~24-71-10092,  https://rscf.ru/en/project/24-71-10092/.

\section{Admissible subcategories}
\label{sec:basics}

In this section we briefly recall the notions of admissible subcategories and semiorthogonal decompositions. A more thorough discussion of those notions can be found, for example, in the book \cite{huybrechts-fm} or in the paper~\cite{bondal-kapranov}.

An \textit{admissible subcategory} of a triangulated category $T$ is a strictly full triangulated subcategory $\cA \subset T$ such that the inclusion functor $\iota\colon \cA \monoarrow T$ has both the left and the right adjoint.

A (strong) \textit{semiorthogonal decomposition} of the category $T$ is a pair of subcategories $\langle \cA, \cB \rangle$ such that:
\begin{itemize}
\item both $\cA$ and $\cB$ are admissible subcategories of $T$;
\item the semiorthogonality holds: for any objects $A \in \cA$ and $B \in \cB$ we have the vanishing~$\RHom_{T}(B, A) = 0$;
\item the smallest triangulated subcategory of $T$ containing both $\cA$ and $\cB$ is $T$.
\end{itemize}
There is a more general notion of a not necessarily strong semiorthogonal decomposition, where the component subcategories are not required to be admissible, but in this paper we only consider semiorthogonal decompositions of smooth projective varieties, which are automatically strong, so we omit the word ``strong''.

Given a semiorthogonal decomposition $T = \langle \cA, \cB \rangle$, for any object $E \in T$ there exists a unique up to a unique isomorphism distinguished triangle in $T$
\[ B \to E \to A \to B[1], \]
where $A \in \cA$ and $B \in \cB$. We call it the \textit{projection triangle of $E$}. The projection triangle is functorial; the objects $B$ and $A$ are isomorphic to $\cB_{R}(E)$ and $\cA_{L}(E)$, respectively, where~$\cB_{R}\colon T \to \cB$ and~$\cA_{L}\colon T \to \cA$ are the right and the left adjoint functors, respectively, to the inclusion functors $\cB \monoarrow T$ and $\cA \monoarrow T$. We call $\cB_{R}$ and $\cA_{L}$ the \textit{right} and the \textit{left projection} functors for the semiorthogonal decomposition $T = \langle \cA, \cB \rangle$.

For an admissible subcategory $\cA \subset T$ in a triangulated category, we define the full subcategory
\[ \emptyperp \cA := \{ B \in T \,\, | \,\, \RHom(B, A) = 0 \,\, \text{for any} \,\, A \in \cA \} \]
of the category~$T$. The inclusion functor $\emptyperp \cA \monoarrow T$ always has a right adjoint, but when~$T$ is a \textit{saturated} category (see \cite{bondal-kapranov}; e.g., the category $\Dbcoh(X)$ for a smooth projective variety~$X$), the category $\emptyperp \cA$ is in fact admissible, and there is a semiorthogonal decomposition~$T = \langle \emptyperp \cA, \cA, \rangle$.

A useful way to define functors between derived categories of coherent sheaves is via \textit{Fourier--Mukai transforms}: given smooth proper varieties $X$ and $Y$ and an object $K \in \Dbcoh(X \times Y)$, the Fourier--Mukai transform with the kernel $K$ is the functor $\Phi_{K}\colon \Dbcoh(X) \to \Dbcoh(Y)$ defined by the formula $\Phi_{K}(E) := \pi_{Y *}(\pi_{X}^* (E) \otimes K)$, where $\pi_X$ and $\pi_Y$ are projections from~$X \times Y$ to~$X$ and~$Y$, respectively.

\section{Objects supported on closed subschemes}
\label{sec:supported_objects}

This section contains a collection of well-known basic results about objects in the derived category supported (set-theoretically) on a closed subset of the ambient variety, as well as some slightly less well-known, but easy to prove, properties of admissible subcategories consisting of objects with non-full support.

We use the following notion of a (set-theoretical) support of an object in the bounded derived category:

\begin{definition}
  \label{def:support_of_an_objects}
  Let $X$ be a scheme, and let $E \in \Dbcoh(X)$ be an object. Then the set-theoretical support $\supp(E)$ is the union $\cup_{i \in \ZZ} \supp(\mathcal{H}^{i}(E))$ of the set-theoretical supports of the cohomology sheaves of $E$.
\end{definition}

We say that an object $E \in \Dbcoh(X)$ (or a coherent sheaf $\mathcal{F} \in \Coh(X)$) is \textit{supported on the subscheme $Z \subset X$} if the set-theoretical support of $E$ (or $\mathcal{F}$, respectively) is contained in $Z$.

One can probe the support by looking at maps from or to skyscraper sheaves:

\begin{lemma}
  \label{lem:probing_support_object}
  Let $X$ be a smooth scheme, and let $E \in \Dbcoh(X)$ be an object. Then a closed point~$p \in X$ lies in $\supp(E)$ if and only if $\RHom_{X}(E, \OO_p) \neq 0$, or equivalently if and only if~$\RHom_{X}(\OO_p, E) \neq 0$, where~$\OO_p$ is the skyscraper sheaf at the point~$p$. Moreover, if~$p \in \supp(E)$, then for some~$i \in \ZZ$ there exists a morphism $E \to \OO_p [-i]$ which induces a nonzero map~$\mathcal{H}^{i}(E) \to \OO_p$ on the level of cohomology sheaves.
\end{lemma}

\begin{proof}
  A standard spectral sequence argument reduces this to the case $E \in \Coh(X)$ of coherent sheaves. See~\cite[Ex.~3.30]{huybrechts-fm}.
\end{proof}

For coherent sheaves there is a well-behaved notion of the scheme-theoretic support: each coherent sheaf set-theoretically supported on a closed subset $Z \subset X$ is isomorphic to a pushforward of a coherent sheaf on some infinitesimal neighborhood of $Z$, and there is an optimal choice of a neighborhood and a unique choice of the lift. For objects in the derived category the situation is more complicated: an object $E$ with $\supp(E) \subset Z$ does lift to some infinitesimal neighborhood of $Z$, but not quite in a unique way. The following two statements describe what happens. They are probably well-known since the SGA time, but I couldn't find a convenient reference preceding Rouquier's paper \cite{rouquier-dimension}.

\begin{proposition}
  \label{prop:derived_category_with_support}
  Let $X$ be a separated Noetherian scheme, and let $Z \subset X$ be a closed subset. Let $\Coh_{Z}(X) \subset \Coh(X)$ be the subcategory of coherent sheaves on $X$ whose set-theoretical support is contained in $Z$. Then the induced functor $D^{b}(\Coh_{Z}(X)) \to D^{b}(\Coh(X))$ between bounded derived categories is fully faithful, and its image is the subcategory of objects~$E \in \Dbcoh(X)$ such that $\supp(E) \subset Z$.
\end{proposition}

\begin{proof}
  This is essentially proved in \cite[Lem.~7.40]{rouquier-dimension}. More formally, since the subcategory
  \[
    \Coh_{Z}(X) \subset \Coh(X)
  \]
  is a Serre subcategory, there are known criteria for the equivalence between $D^{b}(\Coh_{Z}(X))$ and the subcategory of $\Dbcoh(X)$ consisting of objects whose cohomology sheaves lie in $\Coh_{Z}(X)$, such as the condition \cite[(1.15), Lemma, (c1)]{keller-exact}, which is verified in the proof of \cite[Lem.~7.40]{rouquier-dimension} using the Artin--Rees Lemma. Note that the case of an affine~$X$ is explicitly mentioned in \cite[(1.15), Example, (b)]{keller-exact} as an application of the general criterion.
\end{proof}

\begin{corollary}
  \label{cor:objects_lift_to_support}
  Let $X$ be a separated Noetherian scheme, and let $Z \subset X$ be a closed subscheme defined by the ideal sheaf $\mathcal{I} \subset \OO_{X}$. For any $n > 0$ let $i_{n}\colon Z_n \monoarrow X$ be the inclusion of the closed subscheme defined by the ideal sheaf $\mathcal{I}^n$, and for $n < m$ let $i_{n, m}\colon Z_n \monoarrow Z_m$ be the natural embedding. Then:
  \begin{itemize}
  \item For any object $E \in \Dbcoh(X)$ with $\supp(E) \subset Z$ there exists some $n \gg 0$ and an object $\overline{E} \in \Dbcoh(Z_n)$ such that $E \iso i_{n *}(\overline{E})$.
  \item For any $n_1, n_2 > 0$ and any two objects $E_1 \in \Dbcoh(Z_{n_1})$, $E_2 \in \Dbcoh(Z_{n_2})$ such that $i_{n_1 *}(E_1) \iso i_{n_2 *}(E_2)$ in $\Dbcoh(X)$ there exists some $N \geq \max(n_1, n_2)$ such that $i_{n_1, N *}(E_1) \iso i_{n_2, N *}(E_2)$ in $\Dbcoh(Z_N)$.
  \end{itemize}
\end{corollary}

\begin{proof}
  We start with the first claim. By Proposition~\ref{prop:derived_category_with_support} any object~$E \in \Dbcoh(X)$ such that~$\supp(E) \subset Z$ is isomorphic to an object in $D^{b}(\Coh_{Z}(X))$, which means that it can be represented by a finite complex 
  \[ 0 \to \mathcal{E}_{a} \to \dots \to \mathcal{E}_{b} \to 0 \]
  of coherent sheaves on $X$ such that each sheaf $\mathcal{E}_{i}$ is supported, set-theoretically, on $Z$. Each coherent sheaf in $\Coh_{Z}(X)$ is scheme-theoretically supported on some finite infinitesimal neighborhood of $Z$, so we may choose $n > 0$ so that any sheaf $\mathcal{E}_{i} \in \Coh_{Z}(X)$ from the finite complex above admits a lift to $Z_n \subset X$. Since the pushforward $i_{n *}\colon \Coh(Z_n) \to \Coh(X)$ is fully faithful on coherent sheaves, the differentials of the complex also lift to $Z_n$ (uniquely), thus producing an object $\overline{E} \in \Dbcoh(Z_n)$ such that $i_{n *}(\overline{E}) \iso E$.

  The second claim is proved similarly: by Proposition~\ref{prop:derived_category_with_support} any morphism $i_{n_1 *}(E_1) \to i_{n_2 *}(E_2)$ in $\Dbcoh(X)$ can be represented by a morphism in the derived category $D^{b}(\Coh_{Z}(X))$, which by the definition of the derived category means a ``roof''
  % #numbered_eq[
    \begin{equation}
      \label{eq:roof_of_lifts}
      i_{n_1 *}(E_1) \leftarrow E^\prime \xrightarrow{\qis} i_{n_2 *}(E_2)
    \end{equation}
  % ]  
  in the homotopy category of bounded complexes of objects from $\Coh_{Z}(X)$. As explained above, we can lift all three (bounded) complexes of coherent sheaves from~\eqref{eq:roof_of_lifts} to some common subscheme $Z_N \subset X$. Since the morphisms in the homotopy category of complexes in $\Coh(X)$ are defined in terms of morphisms in the abelian category $\Coh(X)$, the full faithfulness of the pushforward functor~$i_{N *}\colon \Coh(Z_N) \to \Coh(X)$ implies that both morphisms in~\eqref{eq:roof_of_lifts} admit a (unique) lift to the homotopy category of complexes in~$\Coh(Z_N)$. The functor~$i_{N *}\colon \Coh(Z_N) \to \Coh(X)$ is exact and conservative, so any lift of a quasi-isomorphism is necessarily a quasi-isomorphism. Thus we have lifted the given isomorphism~$i_{n_1 *}(E_1) \to i_{n_2 *}(E_2)$ in $\Dbcoh(X)$ to some morphism $i_{n_1, N *}(E_1) \to i_{n_2, N *}(E_2)$ in the derived category $\Dbcoh(Z_N)$. Again, since $i_{N *}$ is exact and conservative on coherent sheaves, any lift of an isomorphism in $\Dbcoh(X)$ to $\Dbcoh(Z_N)$ is automatically an isomorphism.
\end{proof}

In this paper $Z$ usually denotes the closed subset of $X$ viewed as a reduced closed subscheme, while $\widetilde{Z}$ will denote a possibly non-reduced subscheme to which objects supported on $Z$ lift.

\begin{remark}
  \label{rem:lifting_not_unique}
  A coherent sheaf on $X$ lifts to a closed subscheme $i\colon \widetilde{Z} \monoarrow X$ if and only if its scheme-theoretic support is contained in $\widetilde{Z}$, and in this case the lift is unique. For objects in the derived category there is no hope for the uniqueness: the cohomology sheaves lift uniquely, but the additional gluing data does not. One basic issue is that $\Coh(\widetilde{Z})$ can have infinite homological dimension even when $X$ itself is smooth, so  a direct sum of two coherent sheaves in $\Coh_{Z}(X)$, shifted to degrees $0$ and $\dim X + 1$, could, in principle, be lifted as a cone of a nontrivial $\Ext^{\dim X + 1}$ class on $\widetilde{Z}$. That's far from the only problem: even if both $\widetilde{Z}$ and $X$ are smooth, the pushforward functor $i_*$ may not be injective on higher $\Ext$'s between coherent sheaves.
\end{remark}

We also use the notion of the (set-theoretical) support of an admissible subcategory. This definition is well-behaved due to the following lemma:

\begin{lemma}
  \label{lem:support_of_subcategory_is_closed}
  Let $X$ be a quasi-compact separated scheme, and let $\cA \subset \Perf(X)$ be an admissible subcategory. Then $\cup_{A \in \cA} \supp(A)$ is a closed subset of $X$.
\end{lemma}

\begin{proof}
  Since $X$ is quasi-compact and separated, the category $\Perf(X)$ has a classical generator~$G$, by \cite[Thm.~3.1.1]{bondal-vdbergh}. Its projection $G_{\cA}$ to $\cA \subset \Dbcoh(X)$ is a classical generator for~$\cA$ since~$\Dbcoh(X) \to \cA$ is essentially surjective. It is easy to check that $\supp(G_{\cA}) = \cup_{A \in \cA} \supp(A)$, so this is a closed subset.
\end{proof}

\begin{definition}
  \label{def:support_subcategory}
  Let $X$ be a quasi-compact separated scheme, and let $\cA \subset \Dbcoh(X)$ be an admissible subcategory. The \textit{(set-theoretical) support} of the subcategory $\cA$, denoted by~$\supp(\cA)$, is the closed subset $\cup_{A \in \cA} \supp(A)$ of $X$.
\end{definition}

Below we use the following basic observation:

\begin{lemma}
  \label{lem:probing_support_of_subcategory}
  Let $X$ be a Noetherian scheme, and let $\cA \subset \Dbcoh(X)$ be an admissible subcategory. Let $p \in X$ be a point. The following are equivalent:
  \begin{itemize}
  \item $p \in \supp(\cA)$;
  \item in the projection triangle $B_p \to \OO_p \to A_p \to B_p [1]$ for the skyscraper sheaf with respect to the semiorthogonal decomposition $\Dbcoh(X) = \langle \cA, \emptyperp \cA \rangle$ we have $p \in \supp(A_p)$.
  \end{itemize}
\end{lemma}

\begin{proof}
  The second condition implies the first one by the definition of $\supp(\cA)$. Let us prove the reverse implication. If $p \not\in \supp(A_p)$, then the morphism $\OO_p \to A_p$ in the projection triangle is necessarily zero. The cone of a zero morphism in a triangulated category is a direct sum, hence $B_p \iso \OO_p \oplus A_p [-1]$. Since $B_p$ and $A_p$ are semiorthogonal objects, this is possible only when $A_p = 0$ and $B_p = \OO_p$. However, if the skyscraper sheaf $\OO_p$ lies in $\emptyperp \cA$, all objects in $\cA = (\emptyperp \cA)^\perp$ have to be semiorthogonal to the skyscraper sheaf $\OO_p$, which by Lemma~\ref{lem:probing_support_object} implies that $\supp(\cA)$ does not contain $p$.
\end{proof}

\begin{lemma}
  \label{lem:support_of_projection_kernel}
  Let $X$ be a smooth proper variety, and let $\cA \subset \Dbcoh(X)$ be an admissible subcategory. Let $K \in \Dbcoh(X \times X)$ be the Fourier--Mukai kernel for the (left) projection functor $\Dbcoh(X) \to \cA$. For any closed subvariety $Z \subset X$ the following are equivalent:
  \begin{itemize}
  \item $\supp(\cA) \subset Z$
  \item $\supp(K) \subset Z \times Z$.
  \end{itemize}
\end{lemma}

\begin{proof}
  This is a consequence of Lemma~\ref{lem:probing_support_of_subcategory} and the definition of Fourier--Mukai transforms. For details, see, for example, \cite[Lem.~2.48]{pirozhkov-delpezzo}.
\end{proof}

Another variation of the same ideas is the following fact.

\begin{lemma}
  \label{lem:isolated_points_imply_skyscrapers}
  Let $X$ be a smooth proper variety, and let $\cA \subset \Dbcoh(X)$ be an admissible subcategory. Let $p \in X$ be a point. The following are equivalent:
  \begin{itemize}
  \item the skyscraper sheaf $\OO_p$ at the point $p$ is an object of $\cA$;
  \item $p \not\in \supp(\emptyperp \cA)$;
  \item there exists some $A \in \cA$ such that $p$ is an isolated point of $\supp(A)$.
  \end{itemize}
  Moreover, these conditions imply that $\supp(\cA) = X$.
\end{lemma}

\begin{proof}
  We will only explain why the third property implies the second and leave the other implications as an exercise to the reader. Assume $A \in \cA$ is an object such that $p \in \supp(A)$ is an isolated point, i.e., the complement $\supp(A) \backslash \{p\}$ is closed. By \cite[Lem.~3.9]{huybrechts-fm} there exists a direct sum decomposition $A \iso A^\prime \oplus A^{\prime \prime}$ such that $\supp(A^\prime) = \{p\}$. Admissible subcategories are closed under direct summands, so by replacing $A$ with $A^\prime$ we may assume that $\supp(A) = \{p\}$ in the statement.
  
  Let $i \in \ZZ$ be the smallest integer such that $\mathcal{H}^{i}(A) \neq 0$. Since the coherent sheaf $\mathcal{H}^{i}(A)$ is supported on a single point $p$, by a variant of Nakayama's lemma there exists a nonzero morphism $\OO_p \to \mathcal{H}^{i}(A)$ from the skyscraper sheaf. The skyscraper sheaf is a simple object, so any nonzero morphism from it is a monomorphism. Since $\mathcal{H}^{i}(A)$ is the leftmost nonzero cohomology sheaf of $A$, there exists a morphism $\mathcal{H}^{i}(A)[-i] \to A$. Denote by~$f$ the composition~$\OO_p [-i] \to \mathcal{H}^{i}(A) [-i] \to A$. Note that the morphism $f\colon \OO_p [-i] \to A$ is nonzero on the level of cohomology sheaves. Now let $B \in \emptyperp \cA$ be any object. Assume that $p \in \supp(B)$ to seek a contradiction. Then by the last claim of Lemma~\ref{lem:probing_support_object} for some shift $j \in \ZZ$ there exists a morphism $B \to \OO_p [-j]$ which is nonzero on the level of cohomology sheaves. Consider the composition
  \[ B \to \OO_p [-j] \xrightarrow{f[-i+j]} A [-i+j]. \]
  By construction this composition $B \to A [-i+j]$ induces a nonzero map $\mathcal{H}^{j}(B) \to \mathcal{H}^{i}(A)$, so it is a nonzero morphism in the derived category. But this contradicts the semiorthogonality of $A$ and $B$. Thus all objects in $\emptyperp \cA$ are supported away from the point $p$, as claimed.
\end{proof}

\begin{corollary}
  \label{cor:sod_has_full_support}
  Let $X$ be a smooth proper variety, and let $\cA \subset \Dbcoh(X)$ be an admissible subcategory. Then either $\supp(\cA) = X$, or $\supp(\emptyperp \cA) = X$.
\end{corollary}

\begin{proof}
  If $\supp(\cA) = Z$ is a proper subset of $X$, then by Lemma~\ref{lem:isolated_points_imply_skyscrapers} the orthogonal subcategory $\emptyperp \cA$ contains all skyscraper sheaves on the points of the open set $U = X \backslash Z$. Since $\supp(\emptyperp \cA)$ is closed by Lemma~\ref{lem:support_of_subcategory_is_closed}, it contains the closure of the dense Zariski-open subset~$U \subset X$, i.e., coincides with $X$.
\end{proof}

Finally, we need the following observation:

\begin{lemma}
  \label{lem:skyscraper_projections_are_zero_in_k0}
  Let $X$ be a smooth projective variety, and let $\cA \subset \Dbcoh(X)$ be an admissible subcategory such that $\supp(\cA)$ is not equal to $X$. Let $p \in \supp(\cA)$ be a point and let $A_p \in \cA$ be the (left) projection of the skyscraper sheaf $\OO_{p} \in \Dbcoh(X)$ to the subcategory $\cA$. Then $A_p$ has at least two nonzero cohomology sheaves.
\end{lemma}

\begin{proof}
  By Lemma~\ref{lem:probing_support_of_subcategory} the object $A_p$ is not zero, so it has at least one nonzero cohomology sheaf. Since any nonzero coherent sheaf has a nonzero class in $K_0(X)$, it is enough to prove that the class
  \[ [ A_p ] = \sum_{i=-\infinity}^{\infinity} (-1)^{i} [ \mathcal{H}^{i}(A_p) ] \]
  of the object $A_p$ in $K_0(X)$ is zero. We do this by comparing the class of $A_p$ in $K_0(X)$ with the classes of the projections of skyscraper sheaves at the points in the complement $X \backslash \supp(\cA)$. The key observation, proved below, is that the class $[ \OO_{p} ] \in K_0(X)$ lies in the abelian group generated by the classes $[ \OO_{q} ] \in K_0(X)$ for points $q \in X \backslash \supp(\cA)$.
  
  Denote the dimension of $X$ by $n$. Let $\OO_{X}(H)$ be an ample line bundle. For a collection 
  \[
    s_1, \dots, s_{n} \in \Gamma(X, \OO_{X}(H))
  \]
  of $n$ global sections of $\OO_{X}(H)$, define the Koszul complex
  \[ \mathrm{Kosz}(s_1, \dots, s_{n}) := [ \Lambda^{n}(\OO_{X}(-H)^{\oplus n}) \to \dots \to \Lambda^{2}(\OO_{X}(-H)^{\oplus n}) \to \OO_{X}(-H)^{\oplus n} \to \OO_{X} ] \]
  by setting the differentials to be substitutions of the global section 
  \[
    (s_1, \dots, s_n) \in \Gamma(X, \OO_{X}(H)^{\oplus n}).
  \]
  A standard property of the Koszul complex is that when the common zero locus of the~$n$ sections~$s_1, \dots, s_n$ has dimension zero (in other words, the section $(s_1, \dots, s_n) \in \Gamma(X, \OO_{X}(H)^{\oplus n})$ is regular), the complex is a locally free resolution for the coherent sheaf 
  \[
    \mathrm{coker} (\OO_{X}(-H)^{\oplus n} \to \OO_{X}).
  \]
  Thus the class
  \[ [ \mathrm{coker} ((s_1, \dots, s_n)\colon \OO_{X}(-H)^{\oplus n} \to \OO_{X}) ] = [ \mathrm{Kosz}(s_1, \dots, s_n) ] = \sum_{i=0}^{n} (-1)^{i} [ \Lambda^{i}(\OO_{X}(-H)^{\oplus n}) ] \]
  of the cokernel sheaf in $K_0(X)$ does not depend on the choice of the sections $s_1, \dots, s_n$ as long as the common zero locus is zero-dimensional.

  After possibly replacing $\OO_{X}(H)$ with a sufficiently high power, we can find two collections of $n$ global sections of $\OO_{X}(H)$:
  \begin{itemize}
  \item a collection $s_1, \dots, s_n \in \Gamma(X, \OO_{X}(H))$ of general sections; the common zero locus is a finite collection of points in $X \backslash \supp(\cA)$;
  \item a collection $t_1, \dots, t_n \in \Gamma(X, \OO_{X}(H))$ such that the corresponding divisors 
  \[
    Z(t_1), \dots, Z(t_n) \subset X
  \] intersect transversely at $p \in \supp(\cA)$ with finitely many other intersection points, all transverse, in the complement $X \backslash \supp(\cA)$.
  \end{itemize}

  The first collection induces an isomorphism
  \[ \mathrm{coker} ((s_1, \dots, s_n)\colon \OO_{X}(-H)^{\oplus n} \to \OO_{X}) \iso \OO_{q_1} \oplus \dots \oplus \OO_{q_m}, \]
  where $q_1, \dots, q_m$ is a finite collection of points in $X \backslash \supp(\cA)$. The second collection induces an isomorphism
  \[ \mathrm{coker} ((t_1, \dots, t_n)\colon \OO_{X}(-H)^{\oplus n} \to \OO_{X}) \iso \OO_{p} \oplus \OO_{q^\prime_1} \oplus \dots \oplus \OO_{q^\prime_{r}}, \]  
  where the points $q^\prime_1, \dots, q^\prime_{r}$ lie in the complement $X \backslash \supp(\cA)$. Since both those direct sums have the same class in $K_0(X)$ by the argument above, we obtain the equality
  % #numbered_eq[
    \begin{equation}
      \label{eq:skyscraper_via_generic_skyscrapers}
      [ \OO_{p} ] = (\sum_{i=1}^{m} [ \OO_{q_i} ]) - (\sum_{i = 1}^{r} [ \OO_{q^\prime_{i}} ])
    \end{equation}
  % ]  
  of classes in $K_0(X)$. Let $\cA_{L}\colon \Dbcoh(X) \to \cA$ be the (left) projection functor. Since the action of $\cA_{L}$ on $K_0(X)$ is well-defined, the formula~\eqref{eq:skyscraper_via_generic_skyscrapers} shows that
  \[ [ A_p ] =^{\mathrm{def}} [ \cA_{L}(\OO_{p}) ] = (\sum_{i=1}^{m} [ \cA_{L}(\OO_{q_i}) ]) - (\sum_{i = 1}^{r} [ \cA_{L}(\OO_{q^\prime_{i}}) ]). \]
  For any point $q \in X \backslash \supp(\cA)$ the projection $\cA_{L}(\OO_{q}) \in \cA$ of the skyscraper sheaf $\OO_q$ at the point $q$ is zero by Lemma~\ref{lem:isolated_points_imply_skyscrapers}. In particular, the classes $[ \cA_{L}(\OO_{q_i}) ]$ and $[ \cA_{L}(\OO_{q^\prime_{i}}) ]$ in $K_0(X)$ are zero for any $i$. Hence $[ A_p ]$ is the zero element in $K_0(X)$, which is exactly what we wanted to prove.
\end{proof}

\section{Infinitely extendable line bundles}
\label{sec:picard}

If $X$ is a scheme and $L$ is an invertible sheaf, tensor multiplication by $L$ is an autoequivalence of the derived category $\Dbcoh(X)$. Autoequivalences of this type are very important for the rest of this paper; however, we need to consider algebraic families of these autoequivalences as the line bundle $L$ varies. To do this properly, in this section we briefly recall the notion of the Picard scheme and some of its basic properties (see \cite{kleiman-picard} or \cite[Ch.~9.2]{neron-models} for a proper exposition). Then we introduce the notion of an \textit{infinitely extendable} line bundle on a closed subscheme of a scheme, which will be very important in the remainder of the paper, and discuss some properties of infinitely extendable line bundles.

\begin{definition}
  \label{def:relative_picard}
  Let $X$ be a scheme over a field $\kk$. Its \textit{Picard group} $\Pic(X)$ is the abelian group of invertible sheaves on $X$ modulo isomorphism. The \textit{relative Picard functor}~$\Pic_{X / \kk}(-)\colon \mathrm{Sch} / \kk \to \mathrm{Ab}$ of $X$ is defined as the quotient
  \[ \Pic_{X / \kk}(Y) := \Pic(X \times Y) / \Pic(Y). \]
\end{definition}

The relative Picard functor is not always representable by a $\kk$-scheme, but its sheafification in the fppf topology is, at least when $X$ is proper. This level of generality is necessary for our purposes since we will usually work with line bundles on non-reduced subschemes.

\begin{theorem}[{\cite[(II.15)]{murre-representability} or \cite[Cor.~9.4.18.3]{kleiman-picard}}]
  \label{thm:picard_exists}
  Let $X$ be a proper scheme over a field~$\kk$. Then there exists a group scheme $\Pic(X)$ which represents the fppf-sheafification of the relative Picard functor. Moreover, $\Pic(X)$ is a union of open quasi-projective subschemes.
\end{theorem}

We call this representing scheme \textit{the Picard scheme}.

A universal line bundle on $X \times \Pic(X)$ exists when the relative Picard functor is already an fppf-sheaf, which is not always true. But by definition of the fppf-sheafification it exists at least fppf-locally:

\begin{lemma}
  \label{lem:poincare_bundle_fppf}
  Let $X$ be a proper scheme over an algebraically closed field $\kk$. For any line bundle~$L$ on~$X$ there exists a scheme $V$ with an fppf morphism $f\colon V \to \Pic(X)$ such that the image of $f$ contains the $\kk$-point $L \in \Pic(X)(\kk)$ and there exists a line bundle $\mathcal{L}$ on the product~$X \times V$ such that for any $\kk$-point $v \in V(\kk)$ the restriction of $\mathcal{L}$ to $X \times \{v\}$ is the line bundle $f(v) \in \Pic(X)(\kk)$.
\end{lemma}

The universal line bundle on the whole product $\Pic^0(X) \times X$ is known to exist when $X$ is ``cohomologically flat in dimension zero'', i.e., when for any scheme $T$ over $\kk$ the morphism
\[
  \OO_T \to \pi_{X \times T \to T *}(\OO_{X \times T})
\]
is an isomorphism (see, e.g., \cite[Th.~9.2.5]{kleiman-picard}). This condition, of course, does not have to hold for the non-reduced scheme $X$ since $H^0(\OO_{X})$ can be larger than $\kk$.

Below we work with line bundles on a closed subscheme $Z$ of some proper scheme $X$. Thanks to the projection formula, line bundles restricted from $X$, or even from a formal neighborhood of $Z$ in $X$, have some special properties. We introduce the following definition which will often be used below:

\begin{definition}
  \label{def:infinitely_extendable_line_bundles}
  Let $X$ be a proper scheme over a field $\kk$, and let $i\colon Z \monoarrow X$ be a closed subscheme defined by the ideal sheaf $\mathcal{I} \subset \OO_X$. We say that a line bundle $L$ on $Z$ is \textit{infinitely extendable} if for any $n > 0$ there exists a line bundle $L_n$ on the closed subscheme $Z_n := \underline{\mathrm{Spec}} \OO_{X} / \mathcal{I}^n$ such that $L \iso (L_n)|_Z$. We define the subgroup scheme $\Pic_{\mathrm{i.e.}}(Z / X) \subset \Pic(Z)$ of infinitely extendable line bundles to be the intersection of the images of restriction maps
  \[ \Pic(Z_n) \to \Pic(Z) \]
  over all $n > 0$, and we set $\Pic^0_{\mathrm{i.e.}}(Z / X)$ to be the connected component of the identity in the scheme~$\Pic_{\mathrm{i.e.}}(Z / X)$.
\end{definition}

Note that the image of each morphism $\Pic(Z_n) \to \Pic(Z)$ is a closed subscheme in~$\Pic(Z)$ by~\cite[Ex.~VI$_B$, Prop.~1.2]{sga3-1} since the Picard schemes are locally of finite type by Theorem~\ref{thm:picard_exists} and the morphism is quasi-compact because it is affine as a restriction map along an infinitesimal extension. Since the intersection of any number of closed subschemes is a closed subscheme, it makes sense to speak of the connected component of the identity in~$\Pic_{\mathrm{i.e.}}(Z / X)$.

Picard groups of infinitesimal thickenings are quite well-understood. In particular, for curves any line bundle is infinitely extendable.

\begin{lemma}[{\cite[Ch.~9.2, Prop.~5]{neron-models}}]
  \label{lem:infinitely_extendable_on_curves}
  Let $X$ be a proper scheme over a field~$\kk$, and let~$i\colon Z \monoarrow X$ be a closed subscheme. Suppose that all irreducible components of $Z_{\mathrm{red}}$ have dimension at most $1$. Then $\Pic_{\mathrm{i.e.}}(Z / X) = \Pic(Z)$.
\end{lemma}

\begin{proof}
  Since $\Pic(Z)$ is a scheme locally of finite type by Theorem~\ref{thm:picard_exists}, it is enough to prove that the closed subscheme $\Pic_{\mathrm{i.e.}}(Z / X) \subset \Pic(Z)$ contains all closed points. For simplicity of notation we only prove that $\Pic_{\mathrm{i.e.}}(Z / X)(\kk) = \Pic(Z)(\kk)$, finite extensions of $\kk$ are handled similarly.
  
  Let $\widetilde{Z} \subset X$ be a closed subscheme defined by the ideal sheaf $\mathcal{I}^2$, where $\mathcal{I}$ is the ideal sheaf of $Z \subset X$. Denote by $\mathcal{J}$ the kernel of the morphism $\OO_{\widetilde{Z}}^* \to \OO_{Z}^*$ of sheaves of abelian groups. Then the cohomology long exact sequence \cite[Tag~0C6R]{stacks-project} includes the fragment
  \[ \Pic(\widetilde{Z})(\kk) \to \Pic(Z)(\kk) \to H^2(Z, \mathcal{J}). \]
  Since $\widetilde{Z}$ is a square-zero extension, the sheaf $\mathcal{J}$ of abelian groups happens to be isomorphic to the coherent sheaf $\mathcal{I} / \mathcal{I}^2$. Thus when $\dim(Z) \leq 1$, the group $H^2(Z, \mathcal{J})$ vanishes, so any line bundle on $Z$ can be extended to $\widetilde{Z}$. Applying the same argument to the closed subscheme $\widetilde{Z}$ and iterating shows that line bundles on $Z$ extend to arbitrary infinitesimal thickening.
\end{proof}

When only some of the irreducible components of $Z_{\mathrm{red}}$ have dimension one, the cohomology group $H^2(Z, \mathcal{J})$ in the proof above may not vanish, but there are still many infinitely extendable line bundles. Below we need the following statement:

\begin{lemma}
  \label{lem:line_bundles_on_component_curves}
  Let $X$ be a proper scheme over a field $\kk$, and let $i\colon Z \monoarrow X$ be a closed subscheme defined by the ideal sheaf $\mathcal{I} \subset \OO_X$. Let $C \subset X$ be a closed subscheme such that:
  \begin{itemize}
  \item $\dim C_{\mathrm{red}} = 1$;
  \item each irreducible component of $C_{\mathrm{red}}$ is an irreducible component of $Z$;
  \item $C$ is contained in the subscheme $Z$, i.e., the ideal sheaf $\mathcal{I}_C$ of $C$ contains $\mathcal{I}$.
  \end{itemize}
  Then the composition $\Pic_{\mathrm{i.e.}}(Z / X)(\kk) \monoarrow \Pic(Z)(\kk) \to \Pic(C)(\kk)$ is a surjective map.
\end{lemma}

\begin{proof}
  Denote by $Z^\prime$ the Zariski closure of the complement $Z_{\mathrm{red}} \backslash C_{\mathrm{red}}$, and let $p_1, \dots, p_r \in C_{\mathrm{red}}$ be the closed points constituting the intersection~$C_{\mathrm{red}} \cap Z^\prime$. Let $L$ be a line bundle on~$C$. We need to construct a line bundle~$L_Z$ on~$Z$ and show that it extends to the subscheme~$Z_n := \underline{\mathrm{Spec}} \OO_{X} / \mathcal{I}^n$ for any~$n > 0$. We do this by an explicit gluing procedure: conceptually, since $C \cap Z^\prime$ is a finite set of points, we can glue $L$ on $C$ with a trivial line bundle on $Z^\prime$ in an arbitrary way to get $L_Z$, and then recreate the gluing for any infinitesimal thickening of $Z$. When $Z$ is a reduced curve, this is more or less contained in the proof of~\cite[Ch.~9.2, Prop.~10]{neron-models}, though the key input is just the fact that $C \cap Z^\prime$ is zero-dimensional. The details, somewhat obscured by the notation, are below.

  Pick some $n > 0$. By Lemma~\ref{lem:infinitely_extendable_on_curves} there exists a line bundle $L_n$ on $C_n := \underline{\mathrm{Spec}}(\OO_{X} / \mathcal{I}_C^n)$ extending $L$. Let $\mathcal{U} = \{ U_i \}_{i \in [1; k]}$ be a finite Zariski-open cover of $C$ such that:
  \begin{itemize}
  \item for each $U \in \mathcal{U}$ the restriction $L_n|_{U}$ is trivial (the notation makes sense since Zariski-open subsets in $C$ are the same as in $C_n$); and
  \item each point $p_i \in C_{\mathrm{red}} \cap Z^\prime$ is contained in exactly one subset $U \in \mathcal{U}$.
  \end{itemize}
  Closed subsets have the induced topology, so we can pick a collection of Zariski-open subsets~$\mathcal{W} = \{ W_i \}_{i \in [1; k]}$ of $X$ such that~$W_i \cap C = U_i$ for all $i \in [1; k]$. We may and will assume that each~$W_i$ contains the complement $X \backslash C_{\mathrm{red}}$, so that~$\mathcal{W}$ is in fact an open cover of~$X$.

  Since no two distinct subsets $W, W^\prime$ in the cover $\mathcal{W}$ contain a common point $p_i \in C_{\mathrm{red}} \cap Z^\prime$, the intersection $W \cap W^\prime \cap Z_{\mathrm{red}}$ is a disjoint union of an open subset in $C_{\mathrm{red}}$ and an open subset in $Z^\prime$. This shows that we can define a line bundle on $Z_n$ by gluing trivial line bundles on the cover $\mathcal{W} \cap Z_n$ using, as a cocycle, the cocycle corresponding to $L_n|_{\mathcal{U}}$ along $C$ and just the identity function along $Z^\prime$. This construction defines a line bundle $L_{Z_n}$ on each $Z_n$ in a way that they all agree with $L$ along $C \subset Z$, which is what we wanted to show.
\end{proof}

\begin{remark}
  The argument above shows that the condition $\dim(C \cap Z^\prime) = 0$ implies that the set~$\Pic_{\mathrm{i.e.}}(Z / X)(\kk)$ surjects onto $\Pic_{\mathrm{i.e.}}(C / X)(\kk)$ regardless of the dimension of $C$.
\end{remark}

\section{$\Pic^0$-invariance for admissible subcategories}
\label{sec:invariance}

By Kawatani--Okawa invariance theorem \cite[Th.~1.4]{kawatani-okawa} for a proper variety $X$ any admissible subcategory~$\cA \subset \Dbcoh(X)$ is invariant under the action of~$\Pic^0(X)$ on the derived category of~$X$. When all objects in the subcategory $\cA$ are supported on a closed subset~$Z \subset X$, there is a much weaker, but still useful, invariance property of~$\cA$ with respect to line bundles on infinitesimal thickenings of~$Z$. In this section we state and prove it as Theorem~\ref{thm:admissible_invariance}. We also apply it to get a refinement of another Kawatani--Okawa theorem concerning the support of an admissible subcategory, see Proposition~\ref{prop:condition_on_support}.

The proof of our invariance statement is conceptually similar to the one from \cite{kawatani-okawa}, but we have to struggle due to the worse properties of the Picard scheme for non-reduced schemes. The core of the proof is the following observation by Kawatani and Okawa, which, informally, means that admissible subcategories are closed under small deformations of objects. See~\cite{kuznetsov-basechange} for the notion of the base change for admissible subcategories.

\begin{proposition}
  \label{prop:kawatani_okawa_rigidity}
  Let $X$ be a proper scheme, and let $\cA \subset \Perf(X)$ be an admissible subcategory. Let $U$ be a quasi-compact separated scheme, and let $\mathcal{E} \in \Perf(X \times U)$ be an object. Then there exists the largest Zariski-open subset $U^\prime \subset U$ such that the base change of~$\mathcal{E}$ to~$\Perf(X \times U^\prime)$ lies in the subcategory $\cA \boxtimes \Perf(U^\prime)$. Moreover, a closed point $u \in U$ lies in $U^\prime$ if and only if the base change $E_u \in \Perf(X \times \{u\})$ of the object $\mathcal{E}$ along the morphism~$\{u\} \monoarrow U$ lies in $\cA$.
\end{proposition}

\begin{proof}
  This is essentially proved in \cite[Lem.~3.10]{kawatani-okawa}: take a classical generator $G$ of the orthogonal subcategory $\emptyperp \cA \subset \Perf(X)$, consider the internal $\RHom$-object 
  \[
    \intRHom_{X \times U}(G \boxtimes \OO_U, \mathcal{E}) \in \Perf(X \times U)
  \]
  and let $\mathcal{H}$ be its pushforward along the projection $X \times U \to U$. Since $X$ is proper, the object~$\mathcal{H}$ is a perfect complex on~$U$. The complement of the support of $\mathcal{H}$ is a Zariski-open subset $U^\prime$. All the other claims follow by the base change along the flat morphism $X \times U \to U$.
\end{proof}

The properties of Picard scheme discussed in Section~\ref{sec:picard} easily imply the following ``local'' version of Kawatani--Okawa invariance theorem:

\begin{lemma}
  \label{lem:twists_locally_remain_in_admissible}
  Let $X$ be a smooth proper variety, and let $i\colon \widetilde{Z} \monoarrow X$ be a closed subscheme. Let $\cA \subset \Dbcoh(X)$ be an admissible subcategory. Assume that an object $\overline{E} \in \Dbcoh(\widetilde{Z})$ satisfies~$i_*(\overline{E}) \in \cA$. Then there exists a Zariski-open subset~$U \subset \Pic^0(\widetilde{Z})$ containing the closed point~$\OO_{\widetilde{Z}} \in \Pic^0(\widetilde{Z})(\kk)$ such that for any line bundle~$L \in U(\kk)$ the object~$i_*(\overline{E} \otimes L)$ lies in~$\cA$.
\end{lemma}

\begin{proof}
  The idea is to apply Proposition~\ref{prop:kawatani_okawa_rigidity} to the family of objects $i_*(\overline{E} \otimes L)$ in $\Dbcoh(X)$, but first we need to make sense of the word ``family''. By Lemma~\ref{lem:poincare_bundle_fppf} there exists some fppf morphism $f\colon V \to \Pic^0(\widetilde{Z})$ such that
  \begin{itemize}
  \item the image of $f$ contains the closed point corresponding to the trivial line bundle $\OO_{\widetilde{Z}} \in \Pic^0(\widetilde{Z})(\kk)$;
  \item there exists a line bundle $\mathcal{L}$ on $\widetilde{Z} \times V$ such that for any $\kk$-point $v \in V(\kk)$ the restriction of $\mathcal{L}$ to $\widetilde{Z} \times \{ v \}$ is the line bundle $f(v) \in \Pic^0(\widetilde{Z})(\kk)$.
  \end{itemize}
  
  The line bundle $\mathcal{L}$ on $\widetilde{Z} \times V$ lets us define the family of twists of the object $\overline{E} \in \Dbcoh(\widetilde{Z})$ as an object in $\Dbcoh(\widetilde{Z} \times V)$ via the Fourier--Mukai transform. Denote by $\mathcal{E} \in \Dbcoh(X \times V)$ the pushforward of that family along the closed embedding $\widetilde{Z} \times V \monoarrow X \times V$. Let us show that in fact $\mathcal{E}$ is a perfect complex on $X \times V$.\footnote{
    Over a field of characteristic zero the group scheme $\Pic^0(\widetilde{Z})$ itself is smooth, but even in this case it's not obvious to me whether we can always choose the fppf map $f\colon V \to \Pic^0(\widetilde{Z})$ with a universal family of line bundles such that the source scheme $V$ is smooth. If this were the case, the perfectness would be automatic.
  }
  Perfectness is a local property, so it is enough to prove that $\mathcal{E}$ is perfect in an open neighborhood of any closed point $(z, v) \in Z \times V \subset X \times V$ in its set-theoretical support. Since $\mathcal{L}$ is a line bundle on $\widetilde{Z} \times V$, there exists some Zariski-open neighborhood $W$ of the point $(z, v) \in \widetilde{Z} \times V$ on which $\mathcal{L}$ is trivial. Closed subsets have induced topology, so there exists some Zariski-open set~$W^\prime \subset X \times V$ such that~$W^\prime \cap (Z \times V) = W$. Then the restriction of $\mathcal{E}$ to~$W^\prime$ is isomorphic to the restriction of the object $i_*(\overline{E}) \boxtimes \OO_{V}$ to~$W^\prime$ which is perfect since $i_*(\overline{E})$ is perfect on~$X$. Thus we have~$\mathcal{E} \in \Perf(X \times V)$.

  By Proposition~\ref{prop:kawatani_okawa_rigidity} the set of points $v \in V$ such that the restriction $\mathcal{E}|_{X \times \{v\}}$ lies in the subcategory $\cA$ is a Zariski-open subset $V^\prime \subset V$. It is not empty since by construction it contains the preimage of the trivial line bundle $\OO_{\widetilde{Z}} \in \Pic^0(\widetilde{Z})(\kk)$ along the map $f\colon V \to \Pic^0(\widetilde{Z})$. Since fppf morphisms are open maps (\cite[Tag~01UA]{stacks-project}) the image $f(V^\prime) \subset \Pic^0(\widetilde{Z})$ is a Zariski-open subset which contains the trivial line bundle. This subset $U = f(V^\prime) \subset \Pic^0(\widetilde{Z})$ by construction satisfies the conditions of the lemma.
\end{proof}

Note that by Corollary~\ref{cor:objects_lift_to_support} any object $E \in \cA$ supported on a closed subset $Z \subset X$ lifts to some infinitesimal thickening $i\colon \widetilde{Z} \monoarrow X$. The group $\Pic^0(\widetilde{Z})(\kk)$ does not act on the object~$E \in \Dbcoh(X)$ itself, it only acts on the lift: if~$E_1, E_2 \in \Dbcoh(\widetilde{Z})$ are two objects such that~$i_*(E_1) \iso E \iso i_*(E_2)$, nothing seems to imply that the object~$i_*(E_1 \otimes L)$ is isomorphic to $i_*(E_2 \otimes L)$ if we are not in the situation where the projection formula applies; in general, I don't even know whether the \textit{images} of the two Kodaira--Spencer maps
\[
  T_{\OO_{\widetilde{Z}}} \Pic^0(\widetilde{Z}) \to \Ext^1(E, E)
\]
coincide. Thus the Zariski-open subset $U \subset \Pic^0(\widetilde{Z})$ constructed in Lemma~\ref{lem:twists_locally_remain_in_admissible} depends on the choice of the lift. For future reference, the following observation about the projection formula is indispensable:

\begin{lemma}
  \label{lem:lifts_and_infinitely_extendable_bundles}
  Let $X$ be a proper variety over the field~$\kk$, and let $i\colon \widetilde{Z} \monoarrow X$ be a closed subscheme. Let~$E_1, E_2 \in \Dbcoh(\widetilde{Z})$ be two objects such that~$i_*(E_1)$ is isomorphic to~$i_*(E_2)$ in~$\Dbcoh(X)$. Then for any infinitely extendable line bundle~$L \in \Pic_{\mathrm{i.e.}}(\widetilde{Z} \backslash X)$ there exists an isomorphism~$i_*(E_1 \otimes L) \isoarrow i_*(E_2 \otimes L)$.
\end{lemma}

\begin{proof}
  By Corollary~\ref{cor:objects_lift_to_support} the isomorphism $i_*(E_1) \iso i_*(E_2)$ in $\Dbcoh(X)$ lifts to an isomorphism of objects in $\Dbcoh(Z_n)$ for some infinitesimal neighborhood $i_n\colon Z_n \monoarrow X$ of $\widetilde{Z}$. Let $i_{1, n}\colon \widetilde{Z} \monoarrow Z_n$ be the embedding. Since $L$ is infinitely extendable, there exists some line bundle $\overline{L}$ in $\Pic(Z_n)$ such that $i_{1, n}^* \overline{L} \iso L$. Note that
  \[ i_*(E_1 \otimes L) \iso  (i_{n} \circ i_{1, n})_{*} (E_1 \otimes i_{1, n}^* \overline{L}) \caniso i_{n *}(i_{1, n *}(E_1) \otimes \overline{L}) \]
  by the projection formula. Thus an isomorphism between $i_{1, n *}(E_1)$ and $i_{1, n *}(E_2)$ in $\Dbcoh(Z_n)$ implies the isomorphism between the twists of those two objects by $\overline{L}$, which implies the desired isomorphism after taking the pushforward $i_{n *}$.
\end{proof}

When an admissible subcategory $\cA \subset \Dbcoh(X)$ is supported, in the sense of Definition~\ref{def:support_subcategory}, on a closed subset $Z \subset X$, there is a version of Lemma~\ref{lem:twists_locally_remain_in_admissible} which holds uniformly for all objects in $\cA$. In order to make sense of that statement we need to know that all objects in the admissible subcategory $\cA \subset \Dbcoh(X)$ satisfying $\supp(\cA) = Z$ admit a lift to a common infinitesimal thickening $\widetilde{Z} \subset X$. The lemma below shows that this is true, and, moreover, we can choose a lift of all objects functorially, though in a non-canonical way.

\begin{lemma}
  \label{lem:supported_subcategories_factor}
  Let $X$ be a smooth proper variety. Let $\cA \subset \Dbcoh(X)$ be an admissible subcategory such that $\supp(\cA) = Z$ is not equal to $X$. Then there exists a closed subscheme $i\colon \widetilde{Z} \monoarrow X$ which is an infinitesimal thickening of $Z$ and a functor $\iota_{\cA, \widetilde{Z}}\colon \cA \to \Dbcoh(\widetilde{Z})$ such that the canonical embedding $\cA \monoarrow \Dbcoh(X)$ factors as $i_* \circ \iota_{\cA, \widetilde{Z}}$.
\end{lemma}

\begin{proof}
  Set $\cB = \emptyperp \cA$ and consider the semiorthogonal decomposition $\Dbcoh(X) = \langle \cA, \cB \rangle$. Since~$X$ is smooth and proper, by \cite[Thm.~7.1]{kuznetsov-basechange} there exists a universal projection triangle, i.e., a distinguished triangle
  \[ \cB_R \to \OO_{\Delta_X} \to \cA_L \to \cB_R [1] \]
  of objects in $\Dbcoh(X \times X)$ such that the Fourier--Mukai transform along the object $\cA_L$ is the projection to the subcategory $\cA$. The support condition $\supp(\cA) = Z$ implies that~$\supp(\cA_L) = Z \times Z \subset X \times X$ by~Lemma~\ref{lem:support_of_projection_kernel}.

  Let $\mathcal{I}$ be the ideal sheaf of $Z \subset X$, and denote by $Z_n$ the closed subscheme defined by the ideal $\mathcal{I}^n$. Note that the subschemes $Z_n \times Z_n \subset X \times X$ form a cofinal subset among infinitesimal neighborhoods of $Z \times Z \subset X \times X$. By~\cite[Lem.~7.40]{rouquier-dimension} there exists some~$N \gg 0$ such that the object~$\cA_L \in \Dbcoh(X \times X)$, set-theoretically supported on~$Z \times Z$, is isomorphic to a pushforward of the object~$\overline{\cA_L}$ from the category~$\Dbcoh(Z_N \times Z_N)$. Let $i\colon \widetilde{Z} \monoarrow X$ denote the inclusion of the closed subscheme $Z_N$ into $X$. Then the Fourier--Mukai transform along the object 
  \[ \cA_L \iso (i \times i)_*(\overline{\cA_L}) \in \Dbcoh(X \times X) \]
  decomposes, by the projection formula, into a composition
  % #numbered_eq[
    \begin{equation}
      \label{eq:fm_kernel_factorization}
      \Dbcoh(X) = \Perf(X) \xrightarrow{i^*} \Perf(\widetilde{Z}) \xrightarrow{\overline{\cA_L}} \Dbcoh(\widetilde{Z}) \xrightarrow{i_*} \Dbcoh(X).
    \end{equation}
  % ]  
  Note that even though $\widetilde{Z}$ is not smooth, it is proper, so the Fourier--Mukai transform along~$\overline{\cA_L} \in \Dbcoh(\widetilde{Z} \times \widetilde{Z})$ is well-defined as a functor from $\Perf(\widetilde{Z})$ to $\Dbcoh(\widetilde{Z})$. The canonical embedding $\cA \subset \Dbcoh(X)$, composed with the sequence of functors~\eqref{eq:fm_kernel_factorization}, is the factorization claimed in the statement.
\end{proof}

\begin{remark}
  The choice of the lift $\overline{\cA_L}$ in $\Dbcoh(\widetilde{Z} \times \widetilde{Z})$ is not canonical, as discussed in Remark~\ref{rem:lifting_not_unique}. Thus the functor $\iota_{\cA, \widetilde{Z}}$ constructed above is not canonical as well. I don't know whether the fact that $\cA_L$ fits into the universal projection triangle helps constraining the choices.
\end{remark}

\begin{remark}
  \label{rem:admissible_lifts_to_nonfull_subcategory}
  The functor $\iota_{\cA, \widetilde{Z}}$ constructed above is faithful but not necessarily full, so $\cA$ does not embed as a full subcategory into $\Dbcoh(\widetilde{Z})$. So Lemma~\ref{lem:supported_subcategories_factor} does not imply that an admissible subcategory $\cA$ supported on $Z$ lifts to an admissible subcategory of some $\widetilde{Z} \subset X$. As such, the study of admissible subcategories of higher-dimensional varieties supported on curves cannot be reduced to the study of admissible subcategories for derived categories of curves, even non-reduced ones.
\end{remark}

We can now state the uniform version of Lemma~\ref{lem:twists_locally_remain_in_admissible}.

\begin{proposition}
  \label{prop:twists_locally_remain_in_admissible}
  Let $X$ be a smooth proper variety. Let $\cA \subset \Dbcoh(X)$ be an admissible subcategory such that $\supp(\cA) = Z$ is not equal to $X$. Let $i\colon \widetilde{Z} \monoarrow X$ and $\iota_{\cA, \widetilde{Z}}\colon \cA \to \Dbcoh(\widetilde{Z})$ be as in Lemma~\ref{lem:supported_subcategories_factor}. Then there exists a Zariski-open subset $U \subset \Pic^0(\widetilde{Z})$ containing the closed point $\OO_{\widetilde{Z}} \in \Pic^0(\widetilde{Z})(\kk)$ such that for any line bundle $L \in U(\kk)$ the image of the composition of functors
  \[ \cA \xrightarrow{\iota_{\cA, \widetilde{Z}}} \Dbcoh(\widetilde{Z}) \xrightarrow{- \otimes L} \Dbcoh(\widetilde{Z}) \xrightarrow{i_*} \Dbcoh(X) \]
  lies in the full subcategory $\cA \subset \Dbcoh(X)$.
\end{proposition}

\begin{proof}
  The category $\cA$ admits a classical generator since the projection functor $\Dbcoh(X) \to \cA$ is essentially surjective and $\Dbcoh(X)$ has a classical generator \cite[Cor.~3.1.2]{bondal-vdbergh}. Let $G \in \cA$ be a classical generator.

  By Lemma~\ref{lem:twists_locally_remain_in_admissible} there exists a Zariski-open subset $U \subset \Pic^0(\widetilde{Z})$ such that $i_*(\iota_{\cA, \widetilde{Z}}(G) \otimes L)$ belongs to the subcategory $\cA \subset \Dbcoh(X)$ for all line bundles $L \in U(\kk)$. It remains to show that the same open subset $U$ works for any object $A \in \cA$, not just for the generator. This is a standard argument: for any $L$ the functor
  \[ A \mapsto i_*(\iota_{\cA, \widetilde{Z}}(A) \otimes L) \]
  is triangulated, so the image of the subcategory $\langle G \rangle$ is contained in the subcategory of~$\Dbcoh(X)$ generated by the object~$i_*(\iota_{\cA, \widetilde{Z}}(G) \otimes L) \in \cA$. By definition $\langle G \rangle = \cA$, which finishes the proof.
\end{proof}

Note that, as in Lemma~\ref{lem:twists_locally_remain_in_admissible}, the collection of line bundles $L$ for which the condition 
\[
  i_*(\iota_{\cA, \widetilde{Z}}(-) \otimes L) \in \cA
\]
holds for all objects in $\cA$ may not be a subgroup in $\Pic^0(\widetilde{Z})$ since in general
\[ \iota_{\cA, \widetilde{Z}}(i_*(\iota_{\cA, \widetilde{Z}}(-) \otimes L)) \neq \iota_{\cA, \widetilde{Z}}(-) \otimes L, \]
those two objects only become isomorphic after $i_*$. This prevents us from deducing the full $\Pic^0$-invariance immediately from the fact that an open neighborhood of the identity element preserves $\cA \subset \Dbcoh(X)$.

Kawatani and Okawa used a mostly equivalent approach to prove the $\Pic^0(X)$-invariance for admissible subcategories in $\Dbcoh(X)$ in \cite[Th.~1.4]{kawatani-okawa}: following their argument, we see that the $\Pic^0(\widetilde{Z})$-invariance of $\cA$ could be deduced if we knew that not just the category $\cA$ has a classical generator, but the category $(i_*)^{-1}(\cA) \subset \Dbcoh(\widetilde{Z})$ of objects whose pushforward lands in $\cA$ has a classical generator as well. However, there is no reason to expect this.

Nevertheless, using the projection formula we can at least show the invariance under the action of infinitely extendable line bundles (Definition~\ref{def:infinitely_extendable_line_bundles}). This is the main theorem of this section. For readability purposes we first state it on the level of $\kk$-points of $\Pic^0$ assuming that $\kk$ is algebraically closed. We give a more versatile, but heavier in terms of notation, statement below in Theorem~\ref{thm:admissible_invariance_abstract}.

\begin{theorem}
  \label{thm:admissible_invariance}
  Let $X$ be a smooth proper variety over an algebraically closed field $\kk$. Let~$Z \subset X$ be a subvariety not equal to $X$, and let~$\cA \subset \Dbcoh(X)$ be an admissible subcategory such that~$\supp(\cA) = Z$. Let $i\colon \widetilde{Z} \monoarrow X$ and $\iota_{\cA, \widetilde{Z}}\colon \cA \to \Dbcoh(\widetilde{Z})$ be as in Lemma~\ref{lem:supported_subcategories_factor}. For a line bundle $L \in \Pic^0(\widetilde{Z})(\kk)$ denote by $T_{L}\colon \cA \to \Dbcoh(X)$ the composition $i_*(\iota_{\cA, \widetilde{Z}}(-) \otimes L)$. Then the set $U$ consisting of all line bundles $L \in \Pic^0(\widetilde{Z})(\kk)$ such that $T_{L}(\cA) \subset \cA$ is:
  \begin{itemize}
  \item non-empty (contains $\OO_{\widetilde{Z}} \in \Pic^0(\widetilde{Z})(\kk)$);
  \item Zariski-open; and
  \item invariant under the action of the connected group $\Pic^0_{\mathrm{i.e.}}(\widetilde{Z} \backslash X)(\kk)$ of infinitely extendable line bundles.
  \end{itemize}
\end{theorem}

\begin{proof}
  By Proposition~\ref{prop:twists_locally_remain_in_admissible} we know that the set $U$, as defined in the statement of the theorem, is non-empty and contains a Zariski-open neighborhood of the identity element in $\Pic^0(\widetilde{Z})$. A basic variation of the argument shows that $U$ is, in fact, Zariski-open in $\Pic^0(\widetilde{Z})$. It remains to prove the last claim.
  
  For brevity denote the group scheme $\Pic^0_{\mathrm{i.e.}}(\widetilde{Z} \backslash X)$ by $G$. Let $W$ be the intersection $U \cap G$. Then $W$ is a non-empty Zariski-open subset of the connected group scheme $G$, so the smallest subgroup of $G$ containing $W$ coincides with the whole $G$. Since $\kk$ is algebraically closed, this implies that the $\kk$-points of $W$ generate the group of $\kk$-points of $G$. Thus to prove the $G(\kk)$-invariance of~$U$ it is enough to show that for any line bundle~$L \in W(\kk)$ we have~$L \cdot U \subset U$.

  Let $L^\prime \in U$ be a line bundle. We want to show that the tensor product $L \otimes L^\prime$ also lies in $U$. Let $A \in \cA$ be any object. By the definition of $U$ it suffices to prove that the object~$T_{L^\prime \otimes L}(A)$ lies in~$\cA$.
  Since both $L$ and $L^\prime$ lie in $U$, we know that the object $T_{L}(T_{L^\prime}(A))$ lies in $\cA$. We claim that the objects $T_{L^\prime \otimes L}(A)$ and $T_{L}(T_{L^\prime}(A))$ are, in fact, isomorphic.

  To see this, recall first that the composition $i_*(\iota_{\cA, \widetilde{Z}}(-))$ is the identity functor $\cA \to \cA$ by the construction of $\iota_{\cA, \widetilde{Z}}$ in Lemma~\ref{lem:supported_subcategories_factor}. Thus the pushforward $i_*$ of the object $\iota_{\cA, \widetilde{Z}}(T_{L^\prime}(A))$ is isomorphic to $T_{L^\prime}(A)$. Since by definition $T_{L^\prime}(A)$ is $i_*(\iota_{\cA, \widetilde{Z}}(A) \otimes L^\prime)$, we observe that the objects
  \[ \iota_{\cA, \widetilde{Z}}(T_{L^\prime}(A)) \quad \text{and} \quad \iota_{\cA, \widetilde{Z}}(A) \otimes L^\prime \]
  in $\Dbcoh(\widetilde{Z})$ become isomorphic in $\Dbcoh(X)$ after the application of the pushforward $i_*$. The other line bundle, $L$, lies in $W \subset \Pic_{\mathrm{i.e.}}(\widetilde{Z} \backslash X)$, so it is an infinitely extendable line bundle. Then by Lemma~\ref{lem:lifts_and_infinitely_extendable_bundles} we conclude that the objects
  \[ \iota_{\cA, \widetilde{Z}}(T_{L^\prime}(A)) \otimes L \quad \text{and} \quad \iota_{\cA, \widetilde{Z}}(A) \otimes L^\prime \otimes L \]
  also become isomorphic after the pushforward $i_*$. But their pushforwards
  \[ i_*(\iota_{\cA, \widetilde{Z}}(T_{L^\prime}(A)) \otimes L) \quad \text{and} \quad i_*(\iota_{\cA, \widetilde{Z}}(A) \otimes L^\prime \otimes L) \]
  are by definition of $T$ isomorphic to
  \[ T_{L}(T_{L^\prime}(A)) \quad \text{and} \quad T_{L^\prime \otimes L}(A), \]
  respectively. The object on the left hand side lies in $\cA$, so we see that~$T_{L^\prime \otimes L}(\cA) \subset \cA$, which implies that~$L^\prime \otimes L$ lies in $U$. As explained above, this implies that the set $U$ is invariant under the action of $\Pic^0_{\mathrm{i.e.}}(\widetilde{Z} \backslash X)$.
\end{proof}

\begin{remark}
  Theorem~\ref{thm:admissible_invariance} is an improved version of Proposition~\ref{prop:twists_locally_remain_in_admissible}, which in turn was built on the ``local'' invariance statement from Lemma~\ref{lem:twists_locally_remain_in_admissible}. The key difference between Proposition~\ref{prop:twists_locally_remain_in_admissible} and Lemma~\ref{lem:twists_locally_remain_in_admissible} is that in Lemma~\ref{lem:twists_locally_remain_in_admissible} we do not require the whole subcategory~$\cA$ to be supported on~$Z$, but work only with a single object of $\cA$ which happens to be supported on~$Z$. This single object version of the statement cannot be improved to claim the invariance under $\Pic^0_{\mathrm{i.e.}}$, or at least the proof above does not work.
\end{remark}

For readers interested in non-algebraically closed fields we state a more general version of Theorem~\ref{thm:admissible_invariance}.

\begin{theorem}
  \label{thm:admissible_invariance_abstract}
  Let $X$ be a smooth proper variety over a field $\kk$. Let $\cA \subset \Dbcoh(X)$ be an admissible subcategory such that $\supp(\cA) = Z$ is not equal to $X$. Let $i\colon \widetilde{Z} \monoarrow X$ and~$\iota_{\cA, \widetilde{Z}}\colon \cA \to \Dbcoh(\widetilde{Z})$ be as in Lemma~\ref{lem:supported_subcategories_factor}. Then there exists the largest Zariski-open subset~$U \subset \Pic^0(\widetilde{Z})$ with the following properties:
  \begin{itemize}
  \item $U$ preserves $\cA$: there exists a surjective fppf morphism $f\colon V \to U$ and a line bundle~$\mathcal{L}$ on $V \times \widetilde{Z}$ universal with respect to the composition $V \to U \monoarrow \Pic^0(\widetilde{Z})$ such that the image of the functor
  \[ (\mathrm{id}_{V} \times i)_* (\pi_2^*(\iota_{\cA, \widetilde{Z}}(-)) \otimes \mathcal{L})\colon \cA \to \Perf(V \times X) \]
  is contained in $\Perf(V) \boxtimes \cA$.
  \item $U$ is large: if $x \in \Pic^0(\widetilde{Z})$ is a closed point corresponding to a line bundle $L$ on $\widetilde{Z}_{K}$ for some field extension $K \supset \kk$ and the composition $(i \otimes K)_*(\iota_{\cA, \widetilde{Z}}(-) \otimes L)\colon \cA \to \Perf(X_K)$ has image in $\cA_K \subset \Perf(X_K)$, then $x \in U$. In particular, $U$ contains the point corresponding to the trivial line bundle~$\OO_{\widetilde{Z}} \in \Pic^0(\widetilde{Z})(\kk)$.
  \item $U$ is invariant under the action of the connected group $\Pic^0_{\mathrm{i.e.}}(\widetilde{Z} \backslash X)$.
  \end{itemize}
\end{theorem}

The proof is essentially the same as for Theorem~\ref{thm:admissible_invariance}, so we leave it to the interested reader.

In the paper \cite{kawatani-okawa} Kawatani and Okawa proved that if a smooth projective variety $X$ has effective canonical bundle, then for any admissible subcategory $\cA \subset \Dbcoh(X)$ either $\supp(\cA)$ or $\supp(\emptyperp \cA)$ is a subset of the base locus of $K_X$. This theorem has several generalizations, and using the invariance proved in Theorem~\ref{thm:admissible_invariance} we can find another one. The statement is unpleasantly technical, but it will be very useful for proving Theorem~\ref{thm:no_nefness_on_support} below.

\begin{proposition}
  \label{prop:condition_on_support}
  Let $X$ be a smooth projective variety, and let $\cA \subset \Dbcoh(X)$ be an admissible subcategory with $\supp(\cA) = Z \neq X$. Then for any closed point $p \in Z$ there exists an infinitesimal neighborhood $Z \subset \widetilde{Z} \subset X$ such that for any infinitely extendable line bundle~$M \in \Pic^0_{\mathrm{i.e.}}(\widetilde{Z} \backslash X)(\kk)$ in the connected component of the Picard scheme each global section of~$K_X|_{\widetilde{Z}} \otimes M$ on $\widetilde{Z}$ vanishes at the (reduced) point $p \in Z \subset \widetilde{Z}$.
\end{proposition}

\begin{proof}
  This is a variation on the themes of \cite[Prop.~3.4]{kawatani-okawa} and  \cite[Th.~3.3]{lin-indecomposability}, with the only technical difference being that we do not have an admissible subcategory in $\Dbcoh(\widetilde{Z})$, only a non-full subcategory $\iota_{\cA, \widetilde{Z}}\colon \cA \monoarrow \Dbcoh(\widetilde{Z})$ which becomes admissible after taking the pushforward $i_*\colon \Dbcoh(\widetilde{Z}) \to \Dbcoh(X)$. The proof is conceptually the same, but we have to organize the argument carefully to avoid running into problems with the non-uniqueness of lifts (see Remark~\ref{rem:lifting_not_unique}).

  First, choose some $i\colon \widetilde{Z} \monoarrow X$ and $\iota_{\cA, \widetilde{Z}}\colon \cA \to \Dbcoh(\widetilde{Z})$ as in Lemma~\ref{lem:supported_subcategories_factor}. Consider the projection triangle for the skyscraper sheaf $\OO_{p} \in \Coh(X)$ along the semiorthogonal decomposition $\Dbcoh(X) = \langle \cA, \emptyperp \cA \rangle$:
  % #numbered_eq[
    \begin{equation}
      \label{eq:condition_on_support_projection_triangle}
      B_p \to \OO_p \to A_p \to B_p [1]
    \end{equation}
  % ]  
  with $A_p \in \cA$ and $B_p \in \emptyperp \cA$. Since all three objects in this triangle are set-theoretically supported on $Z \subset X$, by Corollary~\ref{cor:objects_lift_to_support}, after possibly replacing $\widetilde{Z}$ with a thicker infinitesimal neighborhood of $Z$, we can assume that the triangle lifts to $\widetilde{Z}$, i.e., there exist a distinguished triangle
  % #numbered_eq[
    \begin{equation}
      \label{eq:condition_on_support_lifted_triangle}
      \overline{B_p} \to \OO_p \to \overline{A_p} \to \overline{B_p}[1]
    \end{equation}
  % ]  
  in $\Dbcoh(\widetilde{Z})$ whose pushforward to $\Dbcoh(X)$ is isomorphic to the projection triangle~\eqref{eq:condition_on_support_projection_triangle} above. Note that replacing the subscheme $\widetilde{Z}$ with its thickening is compatible with the application of Lemma~\ref{lem:supported_subcategories_factor}, that is, we can define the ``new'' $\iota_{\cA, \widetilde{Z}}$ functor by composing the original functor with the pushforward to the thicker infinitesimal neighborhood, and this composition still enjoys all the properties mentioned in Lemma~\ref{lem:supported_subcategories_factor}.

  Now the proof proceeds by contradiction. Assume that $M \in \Pic^0_{\mathrm{i.e.}}(\widetilde{Z} \backslash X)$ is an infinitely extendable line bundle with a section $s \in \Gamma(\widetilde{Z}, K_X|_{\widetilde{Z}} \otimes M)$ which does not vanish on the reduced point $p \in Z \subset \widetilde{Z}$. Using the lift~\eqref{eq:condition_on_support_lifted_triangle} of the projection triangle we can define the morphism of distinguished triangles in $\Dbcoh(X)$ induced by the multiplication with the section $s\colon M^\dual \to K_X|_{\widetilde{Z}}$:

  % https://q.uiver.app/#q=WzAsOCxbMCwwLCJpXyooXFxvdmVybGluZXtCX3B9IFxcb3RpbWVzIE1eXFxkdWFsKSJdLFsxLDAsImlfKihcXE9PX3AgXFxvdGltZXMgTV5cXGR1YWwpIl0sWzIsMCwiaV8qKFxcb3ZlcmxpbmV7QV9wfSBcXG90aW1lcyBNXlxcZHVhbCkiXSxbMywwLCJpXyooXFxvdmVybGluZXtCX3B9IFxcb3RpbWVzIE1eXFxkdWFsKSBbMV0iXSxbMCwxLCJpXyooXFxvdmVybGluZXtCX3B9KSBcXG90aW1lcyBLX1giXSxbMSwxLCJcXE9PX3twfSBcXG90aW1lcyBLX1giXSxbMiwxLCJpXyooXFxvdmVybGluZXtBX3B9KSBcXG90aW1lcyBLX1giXSxbMywxLCJpXyooXFxvdmVybGluZXtCX3B9KSBcXG90aW1lcyBLX1ggWzFdIl0sWzAsMV0sWzEsMl0sWzIsM10sWzQsNV0sWzUsNl0sWzYsN10sWzMsN10sWzIsNl0sWzEsNV0sWzAsNF0sWzIsNywiIiwxLHsic3R5bGUiOnsiYm9keSI6eyJuYW1lIjoiZGFzaGVkIn19fV1d
  \begin{equation}
    \label{diag:morphism_of_triangles}
    \begin{tikzcd}[cramped]
	{i_*(\overline{B_p} \otimes M^\dual)} & {i_*(\OO_p \otimes M^\dual)} & {i_*(\overline{A_p} \otimes M^\dual)} & {i_*(\overline{B_p} \otimes M^\dual) [1]} \\
	{i_*(\overline{B_p}) \otimes K_X} & {\OO_{p} \otimes K_X} & {i_*(\overline{A_p}) \otimes K_X} & {i_*(\overline{B_p}) \otimes K_X [1]}
	\arrow[from=1-1, to=1-2]
	\arrow[from=1-1, to=2-1]
	\arrow[from=1-2, to=1-3]
	\arrow[from=1-2, to=2-2]
	\arrow[from=1-3, to=1-4]
	\arrow[from=1-3, to=2-3]
	\arrow[dashed, from=1-3, to=2-4]
	\arrow[from=1-4, to=2-4]
	\arrow[from=2-1, to=2-2]
	\arrow[from=2-2, to=2-3]
	\arrow[from=2-3, to=2-4]
  \end{tikzcd}\end{equation}

  Note that in the bottom row we have used the projection formula to pull out the twist by the canonical line bundle. The dashed arrow, defined by the commutativity of the square, is an element of
  \[ \begin{aligned} \RHom(i_*(\overline{A_p} & \otimes M^\dual), i_*(\overline{B_p}) \otimes K_X [1]) \caniso^{\text{Serre duality on X}} \\
  & \caniso \RHom(i_*(\overline{B_p}), i_*(\overline{A_p} \otimes M^\dual)[n - 1])^\dual \caniso^{\text{by definition}} \\
  & \caniso \RHom(B_p, i_*(\overline{A_p} \otimes M^\dual)[n-1])^\dual . \end{aligned} \]
  Since the line bundle $M$ is infinitely extendable, so is its dual $M^\dual \in \Pic^0(\widetilde{Z})$. Then by the last claim of Theorem~\ref{thm:admissible_invariance} we know that $i_*(\overline{A_p} \otimes M^\dual)$ is an object of the subcategory $\cA$. In particular, the semiorthogonality of $\cA$ and $B_p \in \emptyperp \cA$ implies that the last space in the chain of isomorphisms is zero. Thus, in particular, the dashed arrow is zero.

  Since the bottom row of the diagram~\eqref{diag:morphism_of_triangles} is a distinguished triangle, the vanishing of the dashed arrow implies that the vertical map from $i_*(\overline{A_p} \otimes M^\dual)$ admits a lift:

  % https://q.uiver.app/#q=WzAsNCxbMSwwLCJpXyooXFxvdmVybGluZXtBX3B9IFxcb3RpbWVzIE1eXFxkdWFsKSJdLFswLDEsIlxcT09fe3B9IFxcb3RpbWVzIEtfWCJdLFsxLDEsImlfKihcXG92ZXJsaW5le0FfcH0pIFxcb3RpbWVzIEtfWCJdLFsyLDEsImlfKihcXG92ZXJsaW5le0JfcH0pIFxcb3RpbWVzIEtfWCBbMV0iXSxbMSwyXSxbMiwzXSxbMCwyXSxbMCwzLCIwIiwwLHsic3R5bGUiOnsiYm9keSI6eyJuYW1lIjoiZGFzaGVkIn19fV0sWzAsMSwiZiIsMix7InN0eWxlIjp7ImJvZHkiOnsibmFtZSI6ImRvdHRlZCJ9fX1dXQ==
  \begin{equation}
    \label{diag:lift_of_morphism}
    \begin{tikzcd}[cramped]
	& {i_*(\overline{A_p} \otimes M^\dual)} \\
	{\OO_{p} \otimes K_X} & {i_*(\overline{A_p}) \otimes K_X} & {i_*(\overline{B_p}) \otimes K_X [1]}
	\arrow["f"', dotted, from=1-2, to=2-1]
	\arrow[from=1-2, to=2-2]
	\arrow["0", dashed, from=1-2, to=2-3]
	\arrow[from=2-1, to=2-2]
	\arrow[from=2-2, to=2-3]
  \end{tikzcd}\end{equation}

  Denote the dotted arrow by $f$. Note that the objects $i_*(\OO_p \otimes M^\dual)$ and $\OO_p \otimes K_X$ appearing in the middle of~\eqref{diag:morphism_of_triangles} are both isomorphic to the skyscraper sheaf $\OO_p$. Since $\Hom(\OO_p, \OO_p) \iso \kk \cdot \mathrm{id}$, the composition
  % #numbered_eq[
    \begin{equation}
      \label{eq:condition_on_support_composition}
      i_*(\OO_p \otimes M^\dual) \to i_*(\overline{A_p} \otimes M^\dual) \xrightarrow{f} \OO_p \otimes K_X
    \end{equation}
  % ]  
  is either an isomorphism or a zero map. We consider these two cases separately.
  
  If the composition~\eqref{eq:condition_on_support_composition} is an isomorphism, then the morphism $f$ provides a splitting for the map $\OO_p \caniso i_*(\OO_p \otimes M^\dual) \to i_*(\overline{A_p} \otimes M^\dual)$, so $\OO_p$ is a direct summand of $i_*(\overline{A_p} \otimes M^\dual)$. As above, since $M^\dual \in \Pic^0_{\mathrm{i.e.}}(\widetilde{Z} \backslash X)$, by Theorem~\ref{thm:admissible_invariance} we know that the object $i_*(\overline{A_p} \otimes M^\dual)$ lies in the subcategory $\cA$. Thus the existence of $f$ shows that $i_*(\overline{A_p} \otimes M^\dual) \in \cA$ has a direct summand isomorphic to the skyscraper sheaf $\OO_p$. Admissible subcategories are closed under taking direct summands, so this implies that $\OO_p \in \cA$. By Lemma~\ref{lem:isolated_points_imply_skyscrapers} the fact that $\cA$ contains a skyscraper sheaf implies that $\supp(\cA) = X$, which is a contradiction with the assumption $\supp(\cA) = Z \neq X$. So this case does not happen.

  If the composition~\eqref{eq:condition_on_support_composition} is zero, then by the commutativity of the triangle~\eqref{diag:morphism_of_triangles} we see that in the commutative square
  % https://q.uiver.app/#q=WzAsNCxbMCwwLCJpXyooXFxPT19wIFxcb3RpbWVzIE1eXFxkdWFsKSJdLFsxLDAsImlfKihcXG92ZXJsaW5le0FfcH0gXFxvdGltZXMgTV5cXGR1YWwpIl0sWzAsMSwiXFxPT197cH0gXFxvdGltZXMgS19YIl0sWzEsMSwiaV8qKFxcb3ZlcmxpbmV7QV9wfSkgXFxvdGltZXMgS19YIl0sWzAsMV0sWzIsM10sWzEsM10sWzAsMl1d
  \[\begin{tikzcd}[cramped]
	{i_*(\OO_p \otimes M^\dual)} & {i_*(\overline{A_p} \otimes M^\dual)} \\
	{\OO_{p} \otimes K_X} & {i_*(\overline{A_p}) \otimes K_X}
	\arrow[from=1-1, to=1-2]
	\arrow[from=1-1, to=2-1]
	\arrow[from=1-2, to=2-2]
	\arrow[from=2-1, to=2-2]
  \end{tikzcd}\]
  the composed diagonal map $i_*(\OO_p \otimes M^\dual) \to A_p \otimes K_X$ is zero. However, the left vertical morphism is an isomorphism since by assumption the section $s$ does not vanish at $p$. Thus the lower horizontal map~$\OO_p \otimes K_X \to A_p \otimes K_X$ is necessarily zero. This morphism is a twist by~$K_X$ of the map $\OO_p \to A_p$ in the projection triangle~\eqref{eq:condition_on_support_projection_triangle} for the skyscraper sheaf $\OO_p$. By the universal property of projection triangles the map~$\OO_p \to A_p$ can only be zero when~$A_p = 0$ and~$B_p \iso \OO_p$, but by Lemma~\ref{lem:isolated_points_imply_skyscrapers} this implies $p \not\in \supp(\cA)$, which is a contradiction since by assumption we had $p \in Z = \supp(\cA)$.

  Thus we observe that both possible options for the composed morphism~\eqref{eq:condition_on_support_composition} are ruled out. So we have reached a contradiction and hence all global sections $s \in \Gamma(\widetilde{Z}, K_X|_{\widetilde{Z}} \otimes M)$ of an infinitely extendable line bundle $M \in \Pic^0_{\mathrm{i.e.}}(\widetilde{Z} \backslash X)$ in the connected component of the Picard scheme of $\widetilde{Z}$ necessarily vanish at all (reduced) points of $Z$. This finishes the proof of the proposition.
\end{proof}

\section{Invariance property and projections of skyscrapers}
\label{sec:rigidity}

In Section~\ref{sec:invariance} we used the Zariski-openness of admissible subcategories (Proposition~\ref{prop:kawatani_okawa_rigidity}) to prove that an admissible subcategory $\cA \subset \Dbcoh(X)$ supported on a closed subset $Z \subset X$ is, roughly speaking, locally preserved under the action of $\Pic^0(\widetilde{Z})$ for some closed subscheme~$\widetilde{Z} \subset X$. In this section we employ similar arguments to show that some specific objects of $\cA$ are held invariant under this not-quite-action.

In order to do that, we need the following rigidity result for projection triangles which is an immediate consequence of Proposition~\ref{prop:kawatani_okawa_rigidity}. Its statement is more complicated than its proof.

\begin{proposition}
  \label{prop:isotriviality_projection_triangles}
  Let $X$ be a proper scheme, let $F \in \Perf(X)$ be a perfect complex on $X$, and let~$\cA \subset \Perf(X)$ be an admissible subcategory. Let $U$ be a quasi-compact separated scheme, and let
  % #numbered_eq[
    \begin{equation}
      \label{eq:family_of_projection_triangles}
      \mathcal{E}^\prime \to \pi_1^* F \to \mathcal{E} \to \mathcal{E}^\prime [1]
    \end{equation}
  % ]  
  be a distinguished triangle in $\Perf(X \times U)$. Assume that for some point $u_0 \in U(\kk)$ the restriction of the triangle~\eqref{eq:family_of_projection_triangles} to $X \times \{u_0\}$ is isomorphic to the projection triangle of $F$, i.e., the object $F_{\cA} := \mathcal{E}|_{X \times \{u_0\}}$ lies in $\cA$ and the object $F_{\emptyperp \cA} := \mathcal{E}^\prime|_{X \times \{u_0\}}$ lies in $\emptyperp \cA$. Then, after possibly replacing $U$ with some Zariski-open neighborhood of $u_0$, the triangle~\eqref{eq:family_of_projection_triangles} is isomorphic to the pullback of the triangle
  \[ F_{\emptyperp \cA} \to F \to F_{\cA} \to F_{\emptyperp \cA}[1] \]
  from $X$.
\end{proposition}

In particular, the family $\mathcal{E}$ of objects is (locally) isotrivial. This, probably, means that for any admissible subcategory $\cA \subset \Dbcoh(X)$ the projection arrow $F \to F_{\cA}$, viewed as a ``point'' in a moduli ``space'' of objects in the under-category $F \backslash \Perf(X)$, lies in a zero-dimensional ``irreducible component'', but removing the quote symbols may be quite nontrivial (though see the marked versions of stacks constructed in \cite{antieau-elmanto} for something roughly related). We don't use this interpretation.

\begin{proof}
  By Proposition~\ref{prop:kawatani_okawa_rigidity}, applied first to $\mathcal{E}^\prime$ and then to $\mathcal{E}$, there exists a Zariski-open subset~$U^\prime \subset U$ containing the point~$u_0$ such that the object~$\mathcal{E}^{\prime}|_{X \times U^\prime}$ lies in $\emptyperp \cA \boxtimes \Perf(U^\prime)$ and the object $\mathcal{E}|_{X \times U^\prime}$ lies in the subcategory $\cA \boxtimes \Perf(U^\prime)$. Then over $U^\prime$ the distinguished triangle~\eqref{eq:family_of_projection_triangles} is the projection triangle for the object $\pi_1^* F$ and the semiorthogonal decomposition
  \[ \Perf(X \times U^\prime) = \langle \cA \boxtimes \Perf(U^\prime), \emptyperp \cA \boxtimes \Perf(U^\prime) \rangle. \]
  At the same time, the pullback of the triangle
  \[ \pi_1^* F_{\emptyperp \cA} \to \pi_1^* F \to \pi_1^* F_{\cA} \to \pi_1^* F_{\emptyperp \cA} [1] \]
  is also a projection triangle for the same semiorthogonal decomposition. The proposition follows from the uniqueness of projection triangles.
\end{proof}

For us the main impact of Proposition~\ref{prop:isotriviality_projection_triangles} is that one often studies the behavior of an admissible subcategory $\cA \subset \Dbcoh(X)$ by looking at the projections to $\cA$ of the skyscraper sheaves at various points of $X$. Since it's relatively easy to deform a map from a skyscraper sheaf, the rigidity of projection triangles is a strong condition on what kind of object the projection of a skyscraper sheaf can be. In particular, if $\OO_{p} \to A_{p}$ is a left projection of a skyscraper sheaf, all deformations of $A_{p}$ which are locally trivial near $p$ are globally trivial in the following precise sense:

\begin{lemma}
  \label{lem:projections_rigid_near_support}
  Let $X$ be a proper scheme, let $F \in \Perf(X)$ be an object, and let $\cA \subset \Perf(X)$ be an admissible subcategory. Let $U$ be a quasi-compact separated scheme, and let $\mathcal{E} \in \Perf(X \times U)$ be a perfect complex. Assume that for some point $u_0 \in U(\kk)$ the restriction $E := \mathcal{E}|_{X \times \{u_0\}}$ is isomorphic to the left projection of $F$ to $\cA$, i.e., $E$ lies in $\cA$ and there exists a morphism~$F \to E$ whose cone lies in~$\emptyperp \cA$. Assume, additionally, that there exists a Zariski-open subset~$W \subset X$ containing~$\supp(F)$ such that~$\mathcal{E}|_{W \times U}$ is isomorphic to~$\pi_1^*(E)|_{W \times U}$. Then, after possibly replacing $U$ with some Zariski-open neighborhood of $u_0$, the object~$\mathcal{E}$ is isomorphic to~$\pi_1^*E$.
\end{lemma}

\begin{proof}
  Let $j\colon W \monoarrow X$ be the open embedding. Since $W$ contains the support of $F$, to define a morphism $\pi_1^* F \to \mathcal{E}$ on $X \times U$ it suffices to look only at the open set $W \times U$. By assumption the restriction~$\mathcal{E}|_{W \times U}$ is isomorphic to the pullback of $E$, so we can follow a chain of homomorphisms
  \[ \begin{aligned} \Hom_{X}(F, E) \iso \Hom_{W}(j^*F, j^*E) & \xrightarrow{\pi_1^*} \Hom_{W \times U}(\pi_1^* j^* F, \pi_1^* j^* E) \iso \\
  & \iso \Hom_{W \times U}(\pi_1^* j^* F, \mathcal{E}|_{W \times U}) \caniso \Hom_{X \times U}(\pi_1^*F, \mathcal{E}) \end{aligned} \]
  to get a morphism $\pi_1^* F \to \mathcal{E}$ and fit it into a distinguished triangle
  \[ \mathcal{E}^\prime \to \pi_1^* F \to \mathcal{E} \to \mathcal{E}^\prime [1], \]
  where $\mathcal{E}^\prime$ is (the shift of) the cone of the map $\pi_1^*F \to \mathcal{E}$. But then we are in the situation of Proposition~\ref{prop:isotriviality_projection_triangles}, thus $\mathcal{E}$ is locally isotrivial.
\end{proof}

The families of objects considered in Lemma~\ref{lem:twists_locally_remain_in_admissible} (or Proposition~\ref{prop:twists_locally_remain_in_admissible}) are locally trivial around any point. Thus we obtain the following invariance result:

\begin{proposition}
  \label{prop:skyscraper_projections_invariance}
  Let $X$ be a smooth proper variety, and let $\cA \subset \Dbcoh(X)$ be an admissible subcategory. Consider the projection triangle of a skyscraper sheaf at a point $p \in X$:
  \[ B_p \to \OO_p \to A_p \to B_p [1], \]
  where $A_p \in \cA$ and $B_p \in \emptyperp \cA$. Assume that $A_p$ is isomorphic to $i_* E$, where $i\colon \widetilde{Z} \monoarrow X$ is a closed subscheme and $E \in \Dbcoh(\widetilde{Z})$. Then there exists a Zariski-open subset $U \subset \Pic^0(\widetilde{Z})$ containing the closed point $\OO_{\widetilde{Z}} \in \Pic^0(\widetilde{Z})(\kk)$ such that for any line bundle $L \in U(\kk)$ the object $i_*(E \otimes L)$ is isomorphic to $A_p$. Moreover, each cohomology sheaf $\mathcal{H}^{m}(A_p) \iso i_* \mathcal{H}^{m}(E)$ is $\Pic^0(\widetilde{Z})$-invariant.
\end{proposition}

\begin{proof}
  The case where $p \not\in \widetilde{Z}$ is trivial since it may only happen when $A_p = 0$, so assume that~$p$ lies in~$\widetilde{Z}$. As in the proof of Lemma~\ref{lem:twists_locally_remain_in_admissible}, consider an fppf morphism $f\colon V \to \Pic^0(\widetilde{Z})$ whose image contains $\OO_{\widetilde{Z}} \in \Pic^0(\widetilde{Z})(\kk)$, with a universal line bundle $\mathcal{L}$ on $\widetilde{Z} \times V$. As in~\cite[Lem.~9.2.9]{kleiman-picard}, after twisting $\mathcal{L}$ by the pullback of a line bundle $\mathcal{L}^\dual|_{\{p\} \times V}$ we can assume that~$\mathcal{L}$ is equipped with a \textit{rigidification at $p$}, i.e. an isomorphism $u\colon \mathcal{L}|_{\{p\} \times V} \to \OO_V$. By the projection formula the map $u$ induces an isomorphism
  \[ \pi_1^* i^* \OO_p \otimes \mathcal{L} \iso \pi_1^* i^* \OO_p \]
  of objects in $\Dbcoh(\widetilde{Z} \times V)$.

  This, in turn, lets us spread out the morphism $\OO_p \to A_p \iso i_* E$ from the projection triangle to the family $i_*(\pi_1^*E \otimes \mathcal{L})$ in the following way:
  \[ \begin{aligned} \Hom_{X}(\OO_p, i_* E) \caniso \Hom_{\widetilde{Z}}(i^* \OO_p, E) \xrightarrow{\pi_1^*} \Hom_{\widetilde{Z} \times V}(\pi_1^* i^* \OO_p, \pi_1^* E) \xrightarrow{\otimes \mathcal{L}} \\
  \to \Hom_{\widetilde{Z} \times V}(\pi_1^* i^* \OO_p \otimes \mathcal{L}, \pi_1^* E \otimes \mathcal{L}) \iso^{u} \Hom_{\widetilde{Z} \times V}(\pi_1^* i^* \OO_p, \pi_1^* E \otimes \mathcal{L}) \isoarrow \\
  \isoarrow \Hom_{X \times V}(\pi_1^* \OO_p, (i \times \mathrm{id}_{V})_*(\pi_1^*E \otimes \mathcal{L})). \end{aligned} \]

  Now we have a morphism $\pi_1^* \OO_p \to (i \times \mathrm{id}_{V})_*(\pi_1^*E \otimes \mathcal{L})$ which specializes to the (scalar multiple of the) morphism $\OO_p \to A_p$ from the projection triangle at any point of $f^{-1}(\{\OO_{\widetilde{Z}}\})$, so by Lemma~\ref{lem:projections_rigid_near_support} we conclude that, after possibly shrinking $V$, there is an isomorphism 
  % #numbered_eq[
    \begin{equation}
      \label{eq:rigidity_of_skyscraper_projections}
      (i \times \mathrm{id}_{V})_*(\pi_1^* E \otimes \mathcal{L}) \iso \pi_1^* A_p .
    \end{equation} 
  % ]  

  Thus for any line bundle $L \in \Pic^0(\widetilde{Z})(\kk)$ in the image of the morphism $f$ we have an isomorphism $i_*(E \otimes L) \iso A_p$, as claimed in the statement. 
  
  It remains to deduce the full $\Pic^0$-invariance for cohomology sheaves. Since the functor $i_*$ is exact, the isomorphism proved above shows that for any cohomology sheaf $\mathcal{H} := \mathcal{H}^{m}(E)$ the pushforward $i_*(\mathcal{H} \otimes L)$ is isomorphic to $i_* \mathcal{H}$ when $L$ lies in a non-empty Zariski-open subset of $\Pic^0(\widetilde{Z})$. Moreover, on coherent sheaves the direct image $i_*\colon \Coh(\widetilde{Z}) \to \Coh(X)$ is conservative, so in fact $\mathcal{H} \otimes L$ is isomorphic to $\mathcal{H}$ in $\Coh(\widetilde{Z})$. The group $\Pic^0(\widetilde{Z})$ acts on the category $\Coh(\widetilde{Z})$, so the stabilizer of the coherent sheaf $\mathcal{H}$ is a subgroup in $\Pic^0(\widetilde{Z})$. In a connected group scheme any subgroup containing a Zariski-open subset is necessarily equal to the whole group.
\end{proof}

\begin{remark}
  \label{rem:partial_equivariance}
  It would be interesting to understand under which conditions the cohomology sheaves $\mathcal{H}^{m}(E)$ are not just $\Pic^0(\widetilde{Z})$-invariant, but admit a structure of a $\Pic^0(\widetilde{Z})$-equivariant coherent sheaf on $\widetilde{Z}$. For example, if~$\Pic^0(\widetilde{Z})$ admits a globally defined universal line bundle~$\mathcal{P}$ on $\widetilde{Z} \times \Pic^0(\widetilde{Z})$ and its rigidification~$u\colon \mathcal{P}|_{\{p\} \times \Pic^0(\widetilde{Z})} \isoarrow \OO_{\Pic^0(\widetilde{Z})}$, considered as a choice of a non-zero vector in the fiber over $p$ for each line bundle in $\Pic^0$, is compatible with the group structure on $\Pic^0(\widetilde{Z})$ in the sense that the tensor product morphism
  % #numbered_eq[
    \begin{equation}
      \label{eq:rigidification_group_homomorphism}
      L|_p \otimes L^\prime|_p \to (L \otimes L^\prime)|_p
    \end{equation}
  % ]  
  sends the tensor product of chosen vectors for $L$ and $L^\prime$ to the chosen vector for $L \otimes L^\prime$, then the isomorphisms $\phi_L\colon \mathcal{H} \otimes L \to \mathcal{H}$ induced from~\eqref{eq:rigidity_of_skyscraper_projections} for each $L$ in a Zariski-open neighborhood $U$ of the identity in $\Pic^0(\widetilde{Z})$ satisfy the cocycle condition $\phi_{L^\prime} \circ (\phi_L \otimes L^\prime) = \phi_{L \otimes L^\prime}$ whenever $L$, $L^\prime$, and $L \otimes L^\prime$ are in $U \subset \Pic^0(\widetilde{Z})$. This collection of isomorphisms is a partially defined $\Pic^0(\widetilde{Z})$-equivariance structure on $\mathcal{H}$ and I expect that it can be canonically extended to the whole $\Pic^0(\widetilde{Z})$ using the fact that $U$ generates $\Pic^0(\widetilde{Z})$. I also expect that the existence of a universal line bundle on $\widetilde{Z} \times \Pic^0(\widetilde{Z})$ is not relevant to this argument, but the gerbiness of the condition~\eqref{eq:rigidification_group_homomorphism} makes me confused.
\end{remark}

We can use this to prove the following corollary which is especially interesting when~$X$ is a surface.

\begin{theorem}
  \label{thm:rational_curves_only}
  Let $X$ be a smooth proper variety over an algebraically closed field $\kk$, and let~$\cA \subset \Dbcoh(X)$ be an admissible subcategory. Denote by $Z$ the support $\supp(\cA)$. If $C \subset Z$ is an irreducible component with $\dim C = 1$, then $C$ has geometric genus $0$, i.e., it is a rational curve.
\end{theorem}

\begin{proof}
  We can perform a sequence of blow-ups of points in $X$ so that the proper transform of $C$ becomes smooth (works in any characteristic since $C$ is just a curve). Since the derived pullback along a blow-up of a point is a fully faithful functor (see, e.g.,~\cite{orlov-blowups}), the subcategory $\cA$ embeds into the derived category of the blow-up, with the support of that copy of $\cA$ being the full preimage of $Z \subset X$, which in particular contains the normalization of $C$ as one of the irreducible components. Thus we can assume that $C$ is smooth, and in this case we need to show that $C$ has genus zero. We prove this by assuming that $g(C) > 0$ and obtaining a contradiction.

  Denote by $Z^\prime$ the Zariski closure of the complement $Z \backslash C$, i.e., the union of all irreducible components of $Z$ except $C$. Let $p \in C$ be a point which does not lie in $Z^\prime$ and consider the projection triangle for the skyscraper sheaf
  % #numbered_eq[
    \begin{equation}
      \label{eq:rational_curves_skyscraper}
      B_p \to \OO_p \to A_p \to B_p [1]
    \end{equation}
  % ]  
  with $A_p \in \cA$ and $B_p \in \emptyperp \cA$. By Corollary~\ref{cor:objects_lift_to_support} there exists some closed subscheme $i\colon \widetilde{Z} \monoarrow X$ such that $A_p \iso i_*(E)$ for some $E \in \Dbcoh(\widetilde{Z})$. By Proposition~\ref{prop:skyscraper_projections_invariance} we know that all cohomology sheaves of $E$ are invariant under the action of $\Pic^0(\widetilde{Z})$. Let $\mathcal{H} := \mathcal{H}^{m}(E) \in \Coh(\widetilde{Z})$ be one of the cohomology sheaves. Consider the non-derived restriction $\mathcal{F} \in \Coh(C)$ of $\mathcal{H}$ along the inclusion $C \monoarrow \widetilde{Z}$. Clearly the sheaf $\mathcal{F}$ is invariant under the twists by the line bundles in the image of the pullback $\Pic^0(\widetilde{Z}) \to \Pic^0(C)$. However, as we proved in Lemma~\ref{lem:line_bundles_on_component_curves}, the pullback map of Picard schemes is surjective in this case: indeed, the lemma shows that any line bundle on $C$ can be extended to an arbitrary infinitesimal neighborhood of $Z = (\widetilde{Z})_{\mathrm{red}}$, which is exactly what we claim here.

  Thus $\mathcal{F}$ is a $\Pic^0(C)$-invariant coherent sheaf on the smooth proper curve $C$. When the curve $C$ has positive genus, by Lemma~\ref{lem:pic_invariant_sheaves_on_smooth_curves} below this happens only when $\mathcal{F}$ is supported on finitely many points of $C$. Note that $\supp(\mathcal{F}) = \supp(\mathcal{H}) \cap C$. Since the finiteness of the set~$\supp(\mathcal{H}) \cap C$ holds for each of the finitely many nonzero cohomology sheaves of~$E \in \Dbcoh(\widetilde{Z})$, we get that $\supp(E) \cap C$ is a finite set of points.  Recall that the point~$p$ was chosen to be in~$C$, but not in~$Z^\prime$. Thus either~$p$ does not lie in the support of $E$ at all, or it is an isolated point of $\supp(E)$. The first option is ruled out by Lemma~\ref{lem:probing_support_of_subcategory} since by assumption $p \in Z$ is a point in the support of $\cA$. The second option leads to a contradiction by Lemma~\ref{lem:isolated_points_imply_skyscrapers}: indeed, if $E \in \cA$ is an object whose support contains $p$ as an isolated point, then $W := \supp(\emptyperp \cA)$ is a closed (Lemma~\ref{lem:support_of_subcategory_is_closed}) subset of $X$ which does not contain $p$. But then both~$Z = \supp(\cA)$ and~$W = \supp(\emptyperp \cA)$ are proper closed subsets of~$X$, which cannot happen by Corollary~\ref{cor:sod_has_full_support}.
\end{proof}

In the proof above we appealed to the following standard fact:

\begin{lemma}
  \label{lem:pic_invariant_sheaves_on_smooth_curves}
  Let $C$ be a smooth proper curve of positive genus over an algebraically closed field $\kk$. If a coherent sheaf $\mathcal{F} \in \Coh(C)$ is invariant under the action of $\Pic^0(C)$, then $\mathcal{F}$ has zero-dimensional support.
\end{lemma}

There is a fancy argument using Fourier--Mukai transform along the Poincaré bundle on the product~$C \times \mathrm{Jac}(C)$, but we rather present a down-to-earth argument.

\begin{proof}
  Consider the determinant of the torsion-free quotient $\mathcal{F} / \mathrm{tors}(\mathcal{F})$. This determinant is a line bundle on $C$ which is invariant under the action of the $(\rk \mathcal{F})$'th powers of all elements from $\Pic^0(C)$. When $\rk \mathcal{F} \neq 0$, the power map $\Pic^0(C) \xrightarrow{\cdot \rk \mathcal{F}} \Pic^0(C)$ is surjective, so the determinant line bundle is in fact a $\Pic^0(C)$-invariant line bundle on $C$. If $\Pic^0(C)$ is not trivial, this is impossible since $\Pic^0(C)$ acts faithfully on $\Pic(C)$. By assumption the genus of~$C$ is positive, so $\Pic^0(C)$ has positive dimension.
\end{proof}

\begin{remark}
  An analogous statement for singular curves is false: for example, if $C$ is a nodal rational curve, the pushforward of the structure sheaf of the normalization $\widetilde{C} \to C$ is invariant under $\Pic^0(C) \iso \GG_m$. It would be interesting and potentially useful for the general goal of understanding which curves can support admissible subcategories to classify $\Pic^0$-invariant coherent sheaves and objects in derived categories for curves which are configurations of possibly singular rational curves.
\end{remark}

\section{Negativity with respect to the canonical class}
\label{sec:negativity_for_canonical}

Above, in Theorem~\ref{thm:rational_curves_only}, we showed that non-rational curves cannot be irreducible components of the support of an admissible subcategory. In this section we prove that a configuration of curves cannot support an admissible subcategory if all curves intersect the canonical class non-negatively. We again rely heavily on the $\Pic^0_{\mathrm{i.e.}}$-invariance of admissible subcategories proved in Section~\ref{sec:invariance}.

To start with, recall the following easy fact:

\begin{lemma}
  \label{lem:no_isolated_points_in_support}
  Let $X$ be a smooth proper variety, and let $\cA \subset \Dbcoh(X)$ be an admissible subcategory. Then $\supp(\cA)$ does not have any isolated points.
\end{lemma}

\begin{proof}
  An easy consequence of Lemma~\ref{lem:isolated_points_imply_skyscrapers}. More geometrically, this can also be obtained by Proposition~\ref{prop:kawatani_okawa_rigidity} by deforming the object via moving the point.
\end{proof}

\begin{theorem}
  \label{thm:no_nefness_on_support}
  Let $X$ be a smooth projective variety over an algebraically closed field $\kk$ of characteristic zero, and let $\cA \subset \Dbcoh(X)$ be an admissible subcategory. Assume that all irreducible components of $Z := \supp(\cA)$ have dimension at most one. If $\cA \neq 0$, then there exists an irreducible component $C \subset Z$ which is a curve such that the intersection number $K_X \cdot C$ is negative, where $K_X$ is the canonical line bundle on $X$.
\end{theorem}

The class of varieties $X$ for which $K_X$ intersects each curve non-negatively is very important in the minimal model program. In this case $K_X$ is a nef line bundle and $X$ is called a \textit{minimal} variety. Let us note the obvious consequence:

\begin{corollary}
  \label{cor:no_curve_supports_when_minimal}
  Let $X$ be a smooth projective variety, and let $\cA \subset \Dbcoh(X)$ be an admissible subcategory. If the canonical linear system $K_X$ is nef, then either $\cA = 0$, or at least one of the irreducible components of $\supp(\cA)$ has dimension at least two.
\end{corollary}

\begin{proof}
  If all irreducible components of $\supp(\cA)$ have dimension one, by Theorem~\ref{thm:no_nefness_on_support} there exists a curve~$C \subset \supp(\cA)$ such that $K_X \cdot C < 0$. But this contradicts the assumption that~$K_X$ is nef.
\end{proof}

Before giving the proof of Theorem~\ref{thm:no_nefness_on_support}, we mention some other simple corollaries.

\begin{corollary}
  \label{cor:nef_with_curve_base_locus}
  Let $X$ be a smooth projective variety. If the canonical bundle $K_X$ is nef and the base locus of $|K_X|$ has dimension at most one, then $X$ admits no non-trivial semiorthogonal decompositions.
\end{corollary}

\begin{proof}
  By Kawatani--Okawa theorem~\cite[Thm.~1.1]{kawatani-okawa} for any semiorthogonal decomposition~$\Dbcoh(X) = \langle \cA, \cB \rangle$ one of the two subcategories, without loss of generality we assume it is the subcategory $\cA$, has $\supp(\cA) \subset \mathrm{Bs} |K_X|$. In particular, $\dim \supp(\cA) \leq 1$. Then Corollary~\ref{cor:no_curve_supports_when_minimal} of Theorem~\ref{thm:no_nefness_on_support} implies that $\cA = 0$ since the canonical bundle $K_X$ intersects any curve non-negatively by assumption.
\end{proof}

In dimension two this observation completes the proof of the following conjecture by Okawa:

\begin{theorem}[{\cite[Conj.~1.8]{okawa-irregular}}]
  \label{thm:minimal_surfaces_indecomposable}
  A smooth projective surface $X$ with nef canonical class~$K_X$ admits a non-trivial semiorthogonal decomposition of~$\Dbcoh(X)$ if and only if~$\OO_X$ is an exceptional object, i.e., when $H^1(\OO_X) = H^2(\OO_X) = 0$.
\end{theorem}

\begin{proof}
  If $\OO_X$ is an exceptional object, then of course $\Dbcoh(X) = \langle \OO_X, \emptyperp \OO_X \rangle$ is a nontrivial semiorthogonal decomposition, so we only need to prove the other implication, namely that~$\Dbcoh(X)$ is indecomposable when either~$H^1(\OO_X)$ or $H^2(\OO_X)$ is nontrivial.

  If $H^2(\OO_X) = 0$, then either $\OO_X$ is an exceptional object, or $H^1(\OO_X) \neq 0$. In the case where~$H^1(\OO_X) \neq 0$ the indecomposability of~$\Dbcoh(X)$ is proved in \cite[Thm.~1.7]{okawa-irregular}. Below we will indicate an alternative proof of this case relying on Theorem~\ref{thm:no_nefness_on_support}, but for now note that only the case of surfaces with $H^2(\OO_X) \neq 0$ remains. By Serre duality this non-vanishing is equivalent to~$H^0(K_X) \neq 0$, in other words, the canonical class $K_X$ is not only nef, but also effective. Then the base locus $\mathrm{Bs} |K_X| \subset X$ is a proper subset of $X$, and hence $\dim \mathrm{Bs} |K_X| \leq 1$. Then by Corollary~\ref{cor:nef_with_curve_base_locus} the category $\Dbcoh(X)$ is semiorthogonally indecomposable category, as claimed.

  The theorem is now proved, but, as promised, let us explain why Theorem~\ref{thm:no_nefness_on_support} handles the case $H^1(\OO_X) \neq 0$ as well. We start by recalling the strategy used in the paper \cite{okawa-irregular}. The key idea is to consider the Albanese morphism $\mathrm{alb}_{X}\colon X \to \mathrm{Alb}(X)$ and its Stein factorization~$X \xrightarrow{f} Y \to \mathrm{Alb}(X)$. Since $\dim \mathrm{Alb}(X) = \dim H^1(\OO_X) > 0$, the morphism is not trivial, and all fibers of the map $f\colon X \to Y$ are either points or curves. By assumption $X$ is a minimal surface, so it is not covered by rational curves, hence a general fiber of $f$ is either a point or a curve of positive genus, in particular the general fiber has an indecomposable derived category. By the stable semiorthogonal indecomposability of $Y$~\cite[Thm.~1.4]{pirozhkov-nssi} we know that any semiorthogonal decomposition of $\Dbcoh(X)$ is $f$-linear, so the indecomposability of the derived category of the general fiber implies that for any semiorthogonal decomposition~$\Dbcoh(X) = \langle \cA, \cB \rangle$ at least one of the components, without loss of generality we assume it's the subcategory $\cA$, is supported on the preimage of some subvariety $Y^\prime \subset Y$ with $\dim Y^\prime < \dim Y$. In particular, the dimension of $\supp(\cA)$ is at most one.
  
  At this point we can again use Corollary~\ref{cor:no_curve_supports_when_minimal} to deduce that $\cA = 0$, which confirms that the category $\Dbcoh(X)$ is indecomposable. Note that in \cite{okawa-irregular} Okawa finishes a proof by performing a detailed study of the curve $\supp(\cA)$ instead: using strong results by Konno on the geometry of minimal surfaces (\cite{konno-singularities}, \cite{konno-fibered}) one can check that $\supp(\cA)$ is a curve contracted by the morphism $f\colon X \to Y$ and, moreover, it is a curve with some very special properties with regard to the $f$-relative base locus of $K_X$.
\end{proof}

Now we return to Theorem~\ref{thm:no_nefness_on_support}.

\begin{proof}[{Proof of Theorem~\ref{thm:no_nefness_on_support}}]
  We prove this by contradiction. Note that $Z$ has no isolated points by Lemma~\ref{lem:no_isolated_points_in_support}, so no zero-dimensional irreducible components. Assume that $Z := \supp(\cA)$ is a one-dimensional closed subset of $X$ with irreducible components $Z = C_1 \cup \dots \cup C_m$, and that $K_X \cdot C_k \geq 0$ for any $k \in [1, m]$.
  
  Let us first show that, given any infinitesimal thickening $i\colon \widetilde{Z} \monoarrow X$ of the subset $Z$, for any point $p \in Z$ we can find a line bundle $M \in \Pic^0(\widetilde{Z})$ and a global section $s \in \Gamma(\widetilde{Z}, K_X|_{\widetilde{Z}} \otimes M)$ such that $s(p) \neq 0$.
  The key input is the structure theory of Picard groups of proper (possibly reducible and non-reduced) curves: by~\cite[Ch.~9.2, Cor.~14]{neron-models} the quotient of~$\Pic(\widetilde{Z})$ by~$\Pic^0(\widetilde{Z})$ embeds into $\ZZ^{\oplus m}$ with the map given by $L \mapsto (\deg_{C_1}(L), \dots, \deg_{C_m}(L))$. As such, to show that the twist of $K_X|_{\widetilde{Z}}$ by some element of $\Pic^0(\widetilde{Z})$ has a global section which is non-vanishing at $p$ it is enough to just find \textit{any} line bundle $L \in \Pic(\widetilde{Z})$ with $\deg_{C_k}(L) = K_X \cdot C_k$ for all $k$, admitting a global section $s \in \Gamma(\widetilde{Z}, L)$ such that $s(p) \neq 0$; then $M := K_X^\dual |_{\widetilde{Z}} \otimes L$ lies in $\Pic^0(\widetilde{Z})$ automatically.
  
  For example, when all intersection numbers $K_X \cdot C_k$ are zero, the trivial line bundle $L = \OO_{\widetilde{Z}}$ works. Since we have assumed that $K_X \cdot C_k \geq 0$ for any $k$, it is enough to show that for each~$k \in [1, m]$ there exists a line bundle $L_k \in \Pic(\widetilde{Z})$ such that $\deg_{C_l}(L_k)$ is $1$ when $l = k$ and $0$ otherwise, with a global section $s_k \in \Gamma(\widetilde{Z}, L_k)$ which does not vanish at $p$. Then the line bundle $L := \bigotimes_{k = 1}^{m} L_k^{\otimes K_X \cdot C_k}$ satisfies the requirements. The existence of $L_k$, as an effective Cartier divisor on $\widetilde{Z}$, is a matter of commutative algebra and is demonstrated in~\cite[Ch.~9.1, Cor.~10]{neron-models}. When $X$ is projective it can also be proved geometrically as follows: take some point $q \in C_k$ not equal to $p$, with $q$ in the smooth part of $Z$. Then, since $X$ is projective and $q \in X$ is a smooth point, there exists some (sufficiently positive) smooth divisor $H \subset X$ which intersects $Z$ at $q$ transversely and does not contain any of the irreducible components of $Z$. Note that $H \cap Z$ is a finite set, in particular $q$ is an isolated point. Then the map $\OO_{X}(-H) \to \OO_X$ restricts to a monomorphism of line bundles on $\widetilde{Z}$, and its cokernel~$\mathcal{F} \in \Coh(\widetilde{Z})$ has a direct summand $\mathcal{F}_q$ supported, set-theoretically, only at~$q \in Z$, and the quotient~$\mathcal{F} / \mathcal{F}_q$ is supported away from $q$. Then the kernel of the composition~$\OO_{X} \to \mathcal{F} \epiarrow \mathcal{F}_q$ is an invertible ideal sheaf since the invertibility is a Zariski-local property. The dual to this kernel is an effective Cartier divisor of degree $1$ along $C_k$ and trivial along the other components of $Z$. This shows that a line bundle $L_k$ with the desired properties exists.

  Now we can finish the proof. Pick some point $p \in Z$. By Proposition~\ref{prop:condition_on_support} there exists some infinitesimal neighborhood $\widetilde{Z}$ of $Z \subset X$ such that any line bundle $M \in \Pic^0_{\mathrm{i.e.}}(\widetilde{Z} \backslash X)$ has the property that all global sections of $K_X|_{\widetilde{Z}} \otimes M$ on $\widetilde{Z}$ vanish at the reduced point~$p \in Z \subset \widetilde{Z}$. However, in the argument above for any infinitesimal neighborhood of $Z$, hence for $\widetilde{Z}$ from Proposition~\ref{prop:condition_on_support} as well, we have explicitly constructed a line bundle $M \in \Pic^0(\widetilde{Z})$ such that $K_X|_{\widetilde{Z}} \otimes M$ has a global section non-vanishing at $p \in Z$. Since $Z$ is a curve, so is its infinitesimal thickening $\widetilde{Z}$, and by Lemma~\ref{lem:infinitely_extendable_on_curves} all line bundles on $\widetilde{Z}$ are automatically infinitely extendable. This is a contradiction, which means that for any admissible subcategory $\cA \subset X$ the support $\supp(\cA)$ can be a curve only when at least one irreducible component of that curve intersects the canonical class negatively.
\end{proof}

\section{Negativity of self-intersections}
\label{sec:self_intersections_negative}

The results of the previous section grew out of thinking about a general question: which (reduced) subvarieties $Z \subset X$ can support an admissible subcategory $\cA \subset \Dbcoh(X)$? Since $Z$ cannot have isolated points (Lemma~\ref{lem:no_isolated_points_in_support}), the simplest non-trivial case is a curve $C$ in the surface $S$. The basic examples of semiorthogonal decompositions of surfaces supported on curves appear from Orlov's blow-up formula: if $C \subset S$ can be contracted to a smooth point, then we can find an admissible subcategory $\cA \subset \Dbcoh(S)$ with $\supp(\cA) = C$. As we will see in Example~\ref{ex:noncontractible_curve} below, one can find other examples, including an admissible subcategory supported on a non-contractible curve $C \subset S$. Nevertheless, the list of known examples of curves supporting an admissible subcategory is short, and it's only slightly longer than the list of curves for which we can prove the non-existence of a supported admissible subcategory.

For example, it follows from the (supported version of the) Kawatani--Okawa indecomposability theorem~\cite[Prop.~3.4]{kawatani-okawa} that a smooth rational curve $C$ in a surface $S$ with self-intersection $C \cdot C = n$ does not admit an admissible subcategory $\cA \subset \Dbcoh(S)$ with~$\supp(\cA) = C$ when $n \leq -2$. In the case when $n = -1$ the structure sheaf $\OO_{C}$ is an exceptional object and thus generates an admissible subcategory supported on $C$, as explained by Orlov's blow-up formula.

To the best of my knowledge, the case of non-negative self-intersections is not fully treated in the literature, and in this section we cover this gap by proving a stronger statement: Theorem~\ref{thm:sod_implies_negativity} below shows that a reduced curve $C \subset S$ in a smooth projective surface cannot support an admissible subcategory if all irreducible components of $C$ have non-negative self-intersections.

\begin{remark}
  Note that the rigidity of admissible subcategories (Proposition~\ref{prop:kawatani_okawa_rigidity}) suggests that the support of an admissible subcategory should not be a movable curve, but this is hard to turn into a proper argument since only the reduced curve moves freely while an admissible subcategory of $\Dbcoh(S)$ supported on $C$ depends on some infinitesimal neighborhood of the curve $C$ (see Lemma~\ref{lem:supported_subcategories_factor}), and this infinitesimal neighborhood, considered as a scheme, may vary non-trivially while the curve is moved in $S$. Deformation theory for admissible subcategories has been studied in the paper~\cite{bor-moduli}, but here one would need a technical generalization to families of non-reduced schemes and not-quite-admissible subcategories (see Remark~\ref{rem:admissible_lifts_to_nonfull_subcategory}), which does not exist so far.
\end{remark}

In order to prove the main theorem of this section, we need to recall some facts about coherent sheaves with non-full support, as well as some special properties of the Euler pairing for such coherent sheaves. 

Recall the notion of the Euler pairing for coherent sheaves:

\begin{definition}
  Let $X$ be a smooth projective variety, and let $\mathcal{F}$, $\mathcal{E}$ be two coherent sheaves on $X$. Their \textit{Euler pairing} is the number
  \[ \chi(\mathcal{F}, \mathcal{E}) = \sum_{i = 0}^{\infinity} (-1)^{i} \dim \Ext^{i}(\mathcal{F}, \mathcal{E}). \]
\end{definition}

On a smooth projective surface all $\Ext$'s of degree higher than two vanish, so in this case we have
\[ \chi(\mathcal{F}, \mathcal{E}) = \dim \Hom(\mathcal{F}, \mathcal{E}) - \dim \Ext^1(\mathcal{F}, \mathcal{E}) + \dim \Ext^2(\mathcal{F}, \mathcal{E}). \]

\begin{lemma}
  \label{lem:support_chi_vanishing}
  Let $X$ be a smooth projective variety. Let $\mathcal{F}$ and $\mathcal{E}$ be two coherent sheaves on~$X$ such that $\supp(\mathcal{F})$ is zero-dimensional and $\supp(\mathcal{E})$ is not equal to $X$. Then the vanishing
  \[
    \chi(\mathcal{F}, \mathcal{E}) = \chi(\mathcal{E}, \mathcal{F}) = 0
  \]
  holds.
\end{lemma}

\begin{remark}
  This is a special property of zero-dimensional supports; just having 
  \[
    \dim \supp(\mathcal{F}) + \dim \supp(\mathcal{E}) < \dim X
  \]
  is not enough for orthogonality.
\end{remark}

\begin{proof}
  The Euler pairing stays constant in families of coherent sheaves. Since on a smooth variety all points have isomorphic infinitesimal thickening, we can move each of the finitely many points of $\supp(\mathcal{F})$ away from the subvariety $\supp(\mathcal{E}) \subset X$, producing a family of coherent sheaves including both the sheaf $\mathcal{F}$ and another coherent sheaf $\mathcal{F}^\prime$ such that
  \[
    \supp(\mathcal{F}^\prime) \cap \supp(\mathcal{E}) = \emptyset
  \]
  and~$\chi(\mathcal{F}, \mathcal{E}) = \chi(\mathcal{F}^\prime, \mathcal{E})$. Since the supports of $\mathcal{F}^\prime$ and $\mathcal{E}$ are disjoint, there are no nontrivial $\Hom$'s or $\Ext$'s between them, which by definition implies that~$\chi(\mathcal{F}^\prime, \mathcal{E}) = 0$.
\end{proof}

\begin{definition}
  Let $R$ be a Noetherian local ring, and let $\mathcal{F}$ be a coherent sheaf on $\mathrm{Spec} R$ supported at the maximal ideal $\{ \mathfrak{m} \} \subset \mathrm{Spec} R$. Then the \textit{length of $\mathcal{F}$} is the length of the longest filtration of coherent subsheaves of $\mathcal{F}$.
\end{definition}

\begin{definition}
  \label{def:length_along_component_of_support}
  Let $X$ be a reduced Noetherian scheme, and let $\mathcal{F}$ be a coherent sheaf on $X$. Let $Z \subset \supp(\mathcal{F})$ be an irreducible component of the set-theoretical support of $\mathcal{F}$. Then the \textit{length of $\mathcal{F}$ along $Z$} is the length of the localization of $\mathcal{F}$ at the generic point of $Z$.
\end{definition}

\begin{example}
  Let $C \subset S$ be a reduced irreducible curve in a smooth projective surface. Then for any $n$ the structure sheaf $\OO_{n C}$ of the $n$'th infinitesimal thickening of $C$ has length $n$ along $C$, and the direct sum $\OO_{C}^{\oplus n}$ also has length $n$ along $C$.
\end{example}

In this section we work with a curve $C \subset S$ in a surface which may have many irreducible components. For convenience of notation we collect the lengths into a single vector:

\begin{definition}
  \label{def:length_vector}
  Let $S$ be a smooth projective surface, and let $C \subset S$ be a reduced curve. Denote the irreducible components of $C$ by $\{ C_i \}_{i \in [1, n]}$. For a coherent sheaf $\mathcal{F}$ on $S$ whose set-theoretical support is contained in $C$ its \textit{length vector}~$\underline{l}(\mathcal{F}) \in \ZZ_{\geq 0}^{n}$ has $i$'th component~$l_{i}(\mathcal{F})$ equal to the length of $\mathcal{F}$ along $C_i$.
\end{definition}

Similarly, we collect the intersection numbers between the irreducible components of $C$ into the \textit{intersection matrix}: if $C = \bigcup_{i \in [1, n]} C_i$ is the decomposition into irreducible components, then the intersection matrix $I \in \mathrm{Mat}_{n}(\ZZ)$ has $(i, j)$'th matrix element equal to the intersection number $C_i \cdot C_j$.

\begin{lemma}
  \label{lem:euler_characteristic_via_intersections}
  Let $S$ be a smooth projective surface, and let $C \subset S$ be a reduced curve. Let~$C = \bigcup_{i \in [1, n]} C_i$ be the decomposition into irreducible components, and let $I \in \mathrm{Mat}_{n}(\ZZ)$ be the intersection matrix of $C$. Let $\mathcal{F}$ and $\mathcal{E}$ be two coherent sheaves supported on $C$. Then the Euler characteristic $\chi(\mathcal{F}, \mathcal{E})$
  is equal to $-I(\underline{l}(\mathcal{F}), \underline{l}(\mathcal{E}))$.
\end{lemma}

\begin{remark}
  Note, in particular, that this implies $\chi(\mathcal{F}, \mathcal{E}) = \chi(\mathcal{E}, \mathcal{F})$, even though Serre duality is usually more complicated.
\end{remark}

\begin{proof}
  First let us consider the special case where both $\mathcal{F}$ and $\mathcal{E}$ are structure sheaves of irreducible components $C_i$ and $C_j$ for some $i, j \in [1, n]$. In this case we need to prove that
  \[ \chi(\OO_{C_i}, \OO_{C_j}) = - C_i \cdot C_j. \]
  Since the Euler pairing is additive with respect to exact sequences, the short exact sequence
  \[ 0 \to \OO_{S}(-C_i) \to \OO_{S} \to \OO_{C_i} \to 0 \]
  implies the equality
  \[ \chi(\OO_{C_i}, \OO_{C_j}) = \chi(\OO_{S}, \OO_{C_j}) - \chi(\OO_{S}, \OO_{C_j} \otimes \OO_{S}(C_i)). \]
  By Riemann--Roch-type argument this number is equal to $- \deg_{C_j}(\OO_{S}(C_i)|_{C_j})$, hence to the intersection number $- C_i \cdot C_j$ by the definition of the intersection product on surfaces (see, e.g., \cite[Th.~V.1.1]{hartshorne-ag}).

  It remains to show that the case of general $\mathcal{F}$ follows from this special case. Since the sheaf $\mathcal{F}$ is set-theoretically supported on $C$, it admits a filtration by subsheaves which are scheme-theoretically supported on $C$. The Euler pairing depends only on the classes of sheaves in $K_0(S)$, so we may assume that the filtration splits, i.e., that $\mathcal{F}$ is a pushforward of a coherent sheaf along the inclusion $j\colon C \monoarrow S$ of the reduced curve $C$.
  
  Let $f_i\colon \widetilde{C_i} \to C_i \monoarrow S$ be the normalization of the irreducible component $C_i \subset C$. Then the morphism $\mathcal{F} \to f_{i *} f_{i}^{*} \mathcal{F}$, where we have for the first time in this paper used the non-derived pullback functor, is an isomorphism at the generic point of $C_i$, i.e., the cokernel has zero-dimensional support while the kernel has length zero along $C_i$. Using the additivity of classes in $K_0(S)$ along exact sequences we may thus assume that $\mathcal{F}$ is a direct sum of zero-dimensional sheaves and pushforwards of coherent sheaves on the smooth curves $\widetilde{C_i}$ along the maps $f_i\colon \widetilde{C_i} \to S$.
  
  For each $i$ a coherent sheaf on the smooth proper curve $\widetilde{C_i}$ is a direct sum of a vector bundle on $\widetilde{C_i}$ and a torsion subsheaf. Thus again we may assume that $\mathcal{F} \iso \bigoplus_{i = 1}^{n} f_{i *}(\mathcal{F}_i) \oplus T$, where $T$ has zero-dimensional support and each $\mathcal{F}_i$ is a vector bundle on $\widetilde{C_i}$. Since we consider the Euler pairing between $\mathcal{F}$ and $\mathcal{E}$, where $\mathcal{E}$ is supported on a curve $C$, by Lemma~\ref{lem:support_chi_vanishing} the direct summand $T \subset \mathcal{F}$ does not contribute to the Euler pairing and will be omitted in the calculations below.

  Since Euler pairing is additive for direct sums, it is enough to consider the basic case where~$\mathcal{F} = f_{i *}(\mathcal{F}_i)$ is isomorphic to the pushforward of a vector bundle on $\widetilde{C_i}$. For any point~$p \in \widetilde{C_i}$ there exists a short exact sequence
  \[ 0 \to \mathcal{F}_i \to \mathcal{F}_{i} \otimes \OO_{\widetilde{C_i}}(p) \to \mathcal{F}_i \otimes \OO_{\{p\}} \to 0, \]
  where the object on the right side is torsion. Thus after possibly replacing $\mathcal{F}_{i}$ with a high enough twist we may assume that the vector bundle $\mathcal{F}_i$ on $\widetilde{C_i}$ is globally generated. Then for any choice of $\rk(\mathcal{F}_{i})$ general global sections of $\mathcal{F}$ the induced morphism $\OO_{\widetilde{C_i}}^{\oplus \rk \mathcal{F}_{i}} \to \mathcal{F}_{i}$ is an isomorphism at the generic point of $\widetilde{C_i}$. Thus we may in fact assume that $\mathcal{F}_i$ is a trivial vector bundle on $\widetilde{C_i}$. Since the pushforward of a trivial bundle along the map $f_{i}\colon \widetilde{C_i} \to S$ differs from the structure sheaf $\OO_{C_i} \in \Coh(S)$ by a zero-dimensional sheaf, again we may assume that $\mathcal{F}$ is isomorphic to $\OO_{C_i}$.

  By performing a similar reduction sequence for the coherent sheaf $\mathcal{E}$ we see that it is enough to check the case where both $\mathcal{F}$ and $\mathcal{E}$ are isomorphic to structure sheaves of some components~$C_i, C_j \subset C$. But this case has already been discussed at the start of the proof.
\end{proof}

\begin{corollary}
  \label{cor:ext_orthogonal_sheaves_are_negative}
  Let $S$ be a smooth projective surface, and let $C \subset S$ be a reduced curve. Let $C = \bigcup_{i \in [1, n]} C_i$ be the decomposition into irreducible components, and let $I \in \mathrm{Mat}_{n}(\ZZ)$ be the intersection matrix of $C$. Let $\mathcal{F}$ and $\mathcal{E}$ be two coherent sheaves supported on $C$. If~$\Ext^1(\mathcal{F}, \mathcal{E}) = 0$, then $I(\underline{l}(\mathcal{F}), \underline{l}(\mathcal{E})) \leq 0$, and the case where $I(\underline{l}(\mathcal{F}), \underline{l}(\mathcal{E})) = 0$ occurs only when $\Ext^{i}(\mathcal{F}, \mathcal{E}) = 0$ for all $i$.
\end{corollary}

\begin{proof}
  On a smooth surface the only negative contribution to the integer $\chi(\mathcal{F}, \mathcal{E})$ is the dimension of $\Ext^1(\mathcal{F}, \mathcal{E})$. The rest follows from Lemma~\ref{lem:euler_characteristic_via_intersections}.
\end{proof}

\begin{lemma}
  \label{lem:ext1_bound_on_surface}
  Let $S$ be a smooth projective surface. For any two objects $E, F \in \Dbcoh(S)$ there is an inequality
  \[ \dim \Ext^1(E, F) \geq \sum_{i = -\infinity}^{\infinity} \dim \Ext^1(\mathcal{H}^{i}(E), \mathcal{H}^{i}(F)). \]
  In particular, if the objects $E$ and $F$ are semiorthogonal, i.e., $\RHom(E, F) = 0$, then for any pair of indices $i, j \in \ZZ$ the space $\Ext^{1}(\mathcal{H}^{i}(E), \mathcal{H}^{j}(F))$ is zero.
\end{lemma}

\begin{proof}
  The spectral sequence computing $\Ext^{\bullet}(E, F)$ in terms of $\Ext$'s between cohomology sheaves (see \cite[(3.1.3.4)]{bbd}) degenerates at the second page since $S$ is a smooth surface, and this implies the desired inequality. For details, see, for example, \cite[Cor.~2.26]{pirozhkov-delpezzo}.
\end{proof}

\begin{theorem}
  \label{thm:sod_implies_negativity}
  Let $S$ be a smooth projective surface, and let $C \subset S$ be a reduced curve. Let~$C = \bigcup_{i \in [1, n]} C_i$ be the decomposition into irreducible components.
  Assume that there exists an admissible subcategory $\cA \subset \Dbcoh(S)$  with $\supp(\cA) = C$. Then at least one of the irreducible components $C_i \subset C$ has negative self-intersection.
\end{theorem}

\begin{remark}
  I suspect that in fact all irreducible components of $C$ have negative self-intersections.
\end{remark}

\begin{proof}
  Let $I \in \mathrm{Mat}_{n}(\ZZ)$ be the intersection matrix of $C$. Note that $I$ is a symmetric matrix which is non-negative away from the diagonal. Thus to show that at least one of the curves has negative self-intersection it is enough to find a vector $v \in \ZZ_{\geq 0}^{n}$ with non-negative integral coefficients such that $I(v, v) < 0$.

  Pick some point $p \in C$. Consider the projection triangle of the skyscraper sheaf at that point:
  % #numbered_eq[
    \begin{equation}
      \label{eq:main_skyscraper_projection}
      B_p \to \OO_p \to A_p \to B_p [1],
    \end{equation}
  % ]  
  where $A_p \in \cA$ and $B_p \in \emptyperp \cA$. Consider the long exact sequence of cohomology sheaves associated to the triangle~\eqref{eq:main_skyscraper_projection}. It splits into a sequence of isomorphisms
  \[ \mathcal{H}^{i}(A_p) \isoarrow \mathcal{H}^{i+1}(B_p) \]
  for $i \neq -1, 0$, and an exact sequence
  \[ 0 \to \mathcal{H}^{-1}(A_p) \to \mathcal{H}^{0}(B_p) \to \OO_p \to \mathcal{H}^{0}(A_p) \to \mathcal{H}^1(B_p) \to 0. \]
  
  Since $\OO_p$ is a simple coherent sheaf, at most one of the two arrows in $\mathcal{H}^{0}(B_p) \to \OO_p \to \mathcal{H}^{0}(A_p)$ is nonzero.\footnote{In fact, it's not hard to show that the morphism $\mathcal{H}^{0}(B_p) \to \OO_p$ is always a surjection when $\supp(\cA) = C$, so the map $\OO_p \to \mathcal{H}^{0}(A_p)$ is zero.} Hence the connecting morphisms
  \[ \phi_{i}\colon \mathcal{H}^{i}(A_p) \to \mathcal{H}^{i+1}(B_p) \]
  are isomorphisms for all but one value of $i \in \ZZ$. By Lemma~\ref{lem:skyscraper_projections_are_zero_in_k0} the object $A_p$ has at least two nonzero cohomology sheaves. Thus there exists some $i \in \ZZ$ such that the cohomology sheaf $\mathcal{H}^{i}(A_p)$ is not zero and the map $\phi_{i}$ is an isomorphism. Then we have
  \[ \dim \Ext^1(\mathcal{H}^{i}(A_p), \mathcal{H}^{i}(A_p)) = \dim \Ext^1(\mathcal{H}^{i+1}(B_p), \mathcal{H}^{i}(A_p)) = 0 \]
  by Lemma~\ref{lem:ext1_bound_on_surface} since the objects $B_p$ and $A_p$ are semiorthogonal, while
  \[ \dim \Hom(\mathcal{H}^{i}(A_p), \mathcal{H}^{i}(A_p)) > 0 \] 
  since this space contains the identity morphism which is nonzero provided that $\mathcal{H}^{i}(A_p) \neq 0$. Thus $\chi(\mathcal{H}^{i}(A_p), \mathcal{H}^{i}(A_p)) > 0$ by definition of the Euler pairing, and by Lemma~\ref{lem:euler_characteristic_via_intersections} this implies that the length vector $\underline{l}(\mathcal{H}^{i}(A_p))$ has a negative square with respect to the intersection matrix~$I$. Thus at least one matrix element of $I$ is strictly negative, which is what we wanted to prove.
\end{proof}

The well-known Orlov's semiorthogonal decomposition for blow-ups produces many examples of admissible subcategories supported on contractible curves. Together with Theorem~\ref{thm:sod_implies_negativity} this may lead one to suspect that a semiorthogonal decomposition supported on a curve in a surface can only exist when the intersection matrix of the curve is negative definite. This is not true, and counterexamples can be constructed using spherical twists along $(-2)$-curves. Here's one counterexample, explicitly.

\begin{example}
  \label{ex:noncontractible_curve}
  Let $K \subset \PP^2$ be a nodal cubic curve. Blow-up the node to get an embedded resolution; note that the proper transform of $K$ is a smooth $\PP^1$ with self-intersection number~$5$. Let $C_1$ be the exceptional divisor. Blow-up one of the two points in the intersection of $C_1$ with the proper transform of $K$, call the exceptional divisor $C_2$. Now we have three smooth rational curves forming a cycle, with self-intersections $4$, $-2$, and $-1$. Blowing up $6$ general points on the proper transform of $K$ we obtain a surface $S$ with two smooth $(-2)$-curves, $\widetilde{K}$ and $C_1$, and one smooth $(-1)$-curve $C_2$, with the intersection matrix of $C := \widetilde{K} \cup C_1 \cup C_2$ given by
  \[ I = \left(\begin{matrix} -2 & 1 & 1 \\ 1 & -2 & 1 \\ 1 & 1 & -1 \end{matrix}\right). \]

  Here we see that $I((1, 1, 1), (1, 1, 1)) = 1 > 0$, so $C$ is not a contractible curve. The structure sheaf of the union $C_1 \cup C_2$ is exceptional since this union can be contracted to a smooth point, the structure sheaf of $\widetilde{K}$ is spherical since $\widetilde{K}$ is a smooth rational $(-2)$-curve, hence the spherical twist of $\OO_{C_1 \cup C_2}$ is an exceptional object as well. This twist is nothing but the middle object in the universal extension
  \[ 0 \to \OO_{C_1 \cup C_2} \to \mathcal{F} \to \OO_{\widetilde{K}}^{\oplus 2} \to 0, \]
  which is a coherent sheaf whose set-theoretical support is $C$. Thus an exceptional object $\mathcal{F}$ generates an admissible subcategory in $\Dbcoh(S)$ whose (set-theoretical) support $C$ is not a contractible curve.

  Note, however, that the length vector of $\mathcal{F}$ is $(2, 1, 1)$ with $I(\underline{l}(\mathcal{F}), \underline{l}(\mathcal{F})) = -1$ being a negative number, as required by Corollary~\ref{cor:ext_orthogonal_sheaves_are_negative} for any exceptional coherent sheaf.
\end{example}

\printbibliography

\end{document}